\RequirePackage{fix-cm}
\documentclass[smallextended]{svjour3}       % onecolumn (second format)
\smartqed  % flush right qed marks, e.g. at end of proof
\usepackage{graphicx}
\usepackage{amsbsy}
\usepackage{amsmath}
\usepackage{amsfonts}
\usepackage{mathrsfs}
\usepackage{mathtools}
\usepackage{bm}
\usepackage{amssymb}
\usepackage{bbm}
\newcount\colveccount
\newcommand*\colvec[1]{
	\global\colveccount#1
	\begin{pmatrix}
		\colvecnext
	}
	\def\colvecnext#1{
		#1
		\global\advance\colveccount-1
		\ifnum\colveccount>0
		\\
		\expandafter\colvecnext
		\else
	\end{pmatrix}
	\fi
}

\newcommand{\ndN}{\mathbb{N}}
\newcommand{\ndZ}{\mathbb{Z}}

\newcommand{\ndR}{\mathbb{R}}

\renewcommand{\Pr}[1]{\mathbb{P}(#1)}

\newcommand{\Prb}[1]{\mathbb{P}\left(#1\right)}

\newcommand{\Ex}[1]{\mathbb{E}[#1]}

\newcommand{\Exb}[1]{\mathbb{E}\left[#1\right]}

\newcommand{\Va}[1]{\mathbb{V}[#1]}

\newcommand{\one}{{\mathbbm{1}}}

\newcommand{\convdis}{\,{\buildrel d \over \longrightarrow}\,}
\newcommand{\convd}{\,{\buildrel d \over \longrightarrow}\,}

\newcommand{\convp}{\,{\buildrel p \over \longrightarrow}\,}

\newcommand{\convas}{\,{\buildrel a.s. \over \longrightarrow}\,}

\newcommand{\eqdist}{\,{\buildrel d \over =}\,}

\newcommand{\cE}{\mathcal{E}}
\newcommand{\cF}{\mathcal{F}}

\newcommand{\cN}{\mathcal{N}}

\newcommand{\cR}{\mathcal{R}}

\newcommand{\mfL}{\mathfrak{L}}

\newcommand{\mfX}{\mathfrak{X}}

\begin{document}

\title{Rerooting multi-type branching trees: the infinite spine case%\thanks{Grants or other notes
}


\author{Benedikt Stufler}%        \and
%        Second Author %etc.
%}


\institute{Benedikt Stufler \at
              Vienna University of Technology \\
              Institute of
Discrete Mathematics and Geometry \\
E-mail: \includegraphics{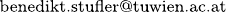}
}


\maketitle

\begin{abstract}
We prove local convergence results of rerooted conditioned multi-type Galton--Watson trees. The limit objects are multitype variants of the random sin-tree constructed by Aldous (1991), and differ according to which types recur infinitely often along the backwards growing spine.
\end{abstract}

\section{Introduction}

The study of multi-type branching processes has received growing attention in recent literature, see~\cite{MR2469338,MR3254703,MR3476213,Deraphelis2017,MR3803914,MR3769811,vatutin_wachtel_2018,AHL_2018__1__149_0}. The reducible case received particular attention in the line of research by by~\cite{MR1601709,MR3254703,MR3431699,MR3488785,MR3568763,MR3646473,MR3860005}. Multi-type Galton--Watson trees are also related to numerous examples of random graphs and  discrete structures, see~\cite{MR2509622,2017arXiv170408714R,StEJC2018,2019arXiv190610355S,quenched}.\footnote{The results of \cite{quenched} were initially part of the present work. The paper was split during the review process following a referee's recommendation.}

In the present work, we prove concentration inequalities for the number of extended fringe subtrees for conditioned multi-type Galton--Watson trees, see Lemma~\ref{le:multfringe}. This allows us to establish  local convergence for conditioned multi-type Galton--Watson trees rerooted at a random location under general assumptions, see Theorems~\ref{te:convfringe},~\ref{te:convfringe2}, and~\ref{te:convfringe3}.
The limit objects are  multi-type generalization of Aldous' invariant random sin-tree, and hence generalize similar objects for monotype trees, see \cite{MR1102319,MR2908619,rerootedstufler2019}. Such trees consist of a root-vertex with an infinite line of ancestors, called the spine, from which further random trees branch off. Depending on the branching mechanism, certain types recur infinitely often along the spine, and others do not. 

We apply our results to the following four settings:
\begin{enumerate}
	\item Sesqui-type trees, whose offspring distribution is critical and has finite variance, conditioned having a total number $n$ of vertices, regardless of their type. See Theorem~\ref{te:daone}.
	\item Reducible multi-type Galton--Watson trees conditioned on having $n$ vertices. See Theorem~\ref{te:transferred}.
	\item Irreducible regular critical multi-type Galton--Watson trees conditioned on the event that a fixed linear combination of the sub-populations by type equals $n$. See Theorem~\ref{te:datwo}.
	\item Irreducible critical multi-type Galton--Watson trees conditioned on having a vector $\bm{k}(n) \in \ndN_0^d$ as total population by types. See Theorem~\ref{te:dathree}.
\end{enumerate}

% In the present work, we give an application to models of random weighted  planar maps. 

%We start by strengthening the annealed local convergence result for face-weighted regular critical planar maps by \cite{MR3769811} to quenched convergence, see Theorem~\ref{te:regcritquench}. The proof goes by showing quenched convergence of the mobiles encoding weighted planar maps to a random infinite mobile with a spine that grows backwards. The bijection by \cite{MR2097335} and the mapping theorem transfer this to a quenched limit for face-weighted planar maps. Hence we use the proof strategy  by \cite{MR3769811} to transfer local convergence of mobiles to local convergence of maps, but instead of starting with a local convergence result for the vicinity of the root of the mobile, we use  quenched convergence for the vicinity of a random vertex in the mobile.

% As an application,  we deduce quenched local convergence of the random planar map $\mM_n^t$ with  $n$ edges and a positive weight $t>0$ at vertices.  That is, $\mM_n^t$ assumes a map $M$ with $n$ edges with probability proportional to $t^{\ve(M)}$, with $\ve(M)$ denoting the number of vertices of $M$. See Theorem~\ref{te:quenchmaptt}. The proof goes by passing to the dual map of a special instance of a random face-weighted map. The vertex weighted random planar map $\mM_n^t$ is related to the study of uniform random planar graphs, see~\cite{MR3068033,MR2735332}. We  apply the  local convergence of $\mM_n^t$  in the subsequent paper~\cite{2019arXiv190804850S} to deduce local convergence of the uniform random planar graph.

\section*{Notation}
We let  $\ndN_0 = \{0,1,2,\ldots\}$ denote the collection of non-negative integers. The random variables in this paper are either canonical or defined some  probability space $(\Omega,\cF,\mathbb{P})$. The law of a random variable $X: \Omega \to S$ with values in some measurable space $S$ is denoted by $\mfL(X)$.  If $Y: \Omega \to S'$ is a random variable with values in some measurable space $S'$, we let $\mfL(X \mid Y)$ denote the conditional law of $X$ given $Y$.  All unspecified limits are taken as $n \to \infty$.  Convergence in probability and distribution are denoted by $\convp$ and $\convd$. Almost sure convergence is denoted by $\convas$.
%For two sequences $(X_n)_n$ and $(Y_n)_n$ of random variables with values in $S$ we write $X_n \atv Y_n$ if their total variation distance $d_{\textsc{TV}}(X_n,Y_n) = \sup_{A \in \mfB(S)} | \Pr{X_n \in A} - \Pr{Y_n \in A}|$ tends to zero.
We say an event holds with high probability if its probability tends to $1$ as $n$ becomes large. 
For any sequence $a_n>0$ we let $o_p(a_n)$ denote an unspecified random variable $Z_n$ such that $Z_n / a_n \convp 0$. Likewise, $O_p(a_n)$ denotes a random variable $Z_n$ such that $Z_n / a_n$ is stochastically bounded. %We say a random variable $X$ is heavy-tailed, if $\Ex{\exp(tX)} = \infty$ for all $t>0$. If $X$ is a random non-negative integer, then this is equivalent to stating that the support $\mathrm{supp}(X) = \{n \in \ndN_0 \mid \Pr{X=n}>0\}$ is infinite and $\Pr{X=n} = \exp(o(n))$ as $n \in \mathrm{supp}(X)$ tends to infinity.

\section*{Index of terminology}

\noindent The following list summarizes frequently used terminology.% in this section.

%\begin{center}
	%	\begin{tabular}{@{}lp{308pt}@{}}	
	\begin{tabular}{p{0.1\linewidth} p{0.75\linewidth}} %{@{}lp{358pt}@{}}
%		asdf & asdf asdf ads fasdf asdf asdf asdf asdf sda asdf asdf asdf asdf asdf asdf asdf asdf \\
		$\mathfrak{G}$ & A collection of vertex types, page~\pageref{de:type}.\\[2pt]
		$\bm{\sigma}$ & An ordered  $\mathfrak{G}$-offspring distribution with $\bm{\sigma} = (\bm{\sigma}_i)_{i \in \mathfrak{G}}$, page~\pageref{de:offspring}.\\[2pt]
		$\bm{\xi}$ & The unordered $\mathfrak{G}$-offspring distribution  $\bm{\xi} = (\bm{\xi}_i)_{i \in \mathfrak{G}}$ corresponding to $\bm{\sigma}$, page~\pageref{de:offspring}.\\[2pt]
		$\#_i(\cdot)$ & Number of vertices of type $i \in \mathfrak{G}$, page~\pageref{de:numi}.\\[2pt]
		$\#(\cdot)$ & Total number of vertices, page~\pageref{de:num}.\\[2pt]
		$\bm{T}$ & A $\bm{\sigma}$-Galton--Watson tree with random root type $\alpha$, page~\pageref{de:bmt}.\\[2pt]
		$\bm{T}^\alpha$ & Like $\bm{T}$, but non-root vertices of type $\kappa$ receive no offspring, page~\pageref{de:ta}.\\[2pt]
		$\bm{T}^\kappa$ & Like $\bm{T}^\alpha$, but root has type $\kappa$, page~\pageref{de:tk}.%\\[2pt]
							\end{tabular}
	
\begin{tabular}{p{0.1\linewidth} p{0.75\linewidth}} %{@{}lp{358pt}@{}}
		$N_T(\cdot)$ & Number of occurrences of some tree $T$ as a fringe subtree, page~\pageref{de:nt}.\\[2pt]
		$\bm{T}(\kappa)$ & Like $\bm{T}$, but root has type~$\kappa$, page~\pageref{de:tkappa}.\\[2pt]
		$\bm{T}_n$ & The tree $\bm{T}$ conditioned on some event $\cE_n$, page~\pageref{de:tn}.\\[2pt]
		$\hat{\bm{T}}^\kappa$ & A random tree with a marked leaf of type~$\kappa$. Distributed like  $\bm{T}^\kappa$ biased by the number of vertices with type~$\kappa$, page~\pageref{de:thk}.\\[2pt]
		$\hat{\bm{T}}(\kappa)$ & A random infinite tree with a marked vertex of type~$\kappa$ and a spine that grows backwards, page~\pageref{de:htk}.\\[2pt]
		$\hat{\bm{T}}^{\kappa,\gamma}$ & A random tree with root type~$\kappa$ and a marked vertex of type $\gamma$. Obtained by biasing $\bm{T}^{\kappa}$ by the number of vertices of type~$\gamma$, page~\pageref{de:tkg}.\\[2pt]
		$\hat{\bm{T}}(\kappa, \gamma)$ & A random infinite tree with a marked vertex of type~$\gamma$ and a spine that grows backwards, page~\pageref{de:tkg2}.
	\end{tabular}        
	
%\end{center}

\vspace{1 \baselineskip}

\section{Basic definitions}

\subsection{Multi-type plane trees}

Suppose we are given a countable non-empty set \label{de:type}$\mathfrak{G}$ whose elements we call \emph{types}. A \emph{$\mathfrak{G}$-type plane tree} is a plane tree together with a map that assigns to each vertex  a type from $\mathfrak{G}$. 

In the present work we only consider trees where each vertex has a finite number of children. Such trees are called \emph{locally finite}. In a locally finite $\mathfrak{G}$-type tree,  the offspring of any vertex is encoded by a word from the free monoid 
\begin{align}
	\langle \mathfrak{G} \rangle = \cup_{n \ge 0} \mathfrak{G}^n.
\end{align}
Here $\mathfrak{G}^0 = \{\emptyset\}$ corresponds to the case where there is no offspring at all.

To any word from $\langle \mathfrak{G} \rangle$ we may assign a vector from the coproduct $\ndN_0^{(\mathfrak{G})}$, that is, the collection of all functions $\mathfrak{G} \to \ndN_0$ that are zero everywhere except for a finite number of types. Here the $j$th coordinate of the vector corresponds to the number of occurrences of a type $j \in \mathfrak{G}$ in the word. (For example, if $\mathfrak{G} = \{a,b,c\}$ is a three-element set, then the word $(a,c,c,a,b,a) \in \langle \mathfrak{G} \rangle$ corresponds to the vector $(x_i)_{i \in \{a,b,c\}} \in \mathbb{N}_0^{(\mathfrak{G})}$ with $x_a = 3$, $x_b = 1$, and $x_c = 2$.) 

For any $\mathfrak{G}$-tree $T$ and any type $i\in \mathfrak{G}$ we let 
\begin{align}
	\label{de:numi}
	\#_i T \in \ndN_0 \cup \{\infty\}
\end{align}  denote the number of vertices of type $i$ in $T$. Furthermore, 
\begin{align}
	\label{de:num}
	\#T := \sum_{i \in \mathfrak{G}} \#_i T
\end{align}
denotes the total number of vertices. We say $T$ is \emph{finite}, if $\# T < \infty$.  %For any subset $\mathfrak{G}_0 \subset \mathfrak{G}$ we set
%\begin{align}
%	\#_{\mathfrak{G}_0} T = \sum_{i \in \mathfrak{G}_0} \#_i T.
%\end{align}

\subsection{Marked trees}

Let $T$ denote a locally finite  $\mathfrak{G}$-type tree and let $v \in T$ be one of its vertices.  We say the tree $T$ is \emph{marked} at $v$. The pair $(T,v)$ is called a \emph{marked tree}.  The directed path from $v$ to the root of $T$ is called the \emph{spine}. The tree $f(T,v)$ formed by $v$ and all its descendants is called the \emph{fringe subtree} of $T$ at $v$. For any $k \ge 0$  we let $f^{[k]}(T,v)$ denote the fringe subtree of the $k$th ancestor of $v$, marked at the vertex corresponding to $v$. (In the case $k=0$, the $k$th ancestor of $v$ is the vertex $v$ itself, and the marked  vertex of $f^{[0]}(T,v)$ coincides with its root vertex.) If the $k$th ancestor of $v$ does not exist, then we set  $f^{[k]}(T,v)$ to some place-holder value. We call $f^{[k]}(T,v)$ an \emph{extended fringe subtree}. See Figure~\ref{fi:efringe} for an illustration.

\begin{figure}[t]
	\centering
	\begin{minipage}{0.8\textwidth}
		\centering
		\includegraphics[width=1.0\textwidth]{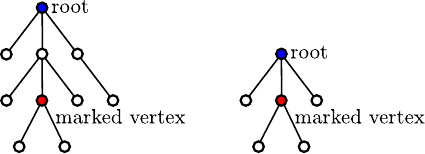}
		\caption{The left-hand side illustrates a finite marked plane tree $(T,v)$, and the right-hand side its extended fringe subtree $f^{[1]}(T,v)$.}
		\label{fi:efringe}
	\end{minipage}
\end{figure}

We may also consider locally finite $\mathfrak{G}$-type plane trees with a marked vertex having a countably infinite number of ancestors. Such a plane tree has hence an infinite backwards growing spine and no root. See~\cite{MR1102319,rerootedstufler2019}  for related notions in the monotype setting. The notion of extended fringe subtrees extends naturally to this setting.

Let $\mfX_\mathrm{f}$ denote the collection of all finite marked $\mathfrak{G}$-type trees. Let $\mfX$ denote the union of $\mfX_\mathrm{f}$ and the collection of  all marked plane trees with an infinite spine such that all extended fringe subtrees are finite. We may endow $\mfX$ with a metric
\begin{align}
	d_{\mfX}(T^\bullet_1, T_2^\bullet) = 
	\begin{cases}
		2,  &f(T_1^\bullet) \neq f(T_2^\bullet) \\
		2^{- \sup\{h \ge 0 \mid f^{[h]}(T_1^\bullet) = f^{[h]}(T_2^\bullet)\}}, &f(T_1^\bullet) = f(T_2^\bullet) 
	\end{cases}, \quad T^\bullet_1, T_2^\bullet \in \mfX.
\end{align}

\begin{proposition}
	The space $(\mfX, d_{\mfX})$ is Polish and $\mfX_\mathrm{f} \subset \mfX$ is a countable dense subset.
\end{proposition}
\begin{proof}
	The space is complete:
	If $(T_n^\bullet)_{n \ge 1}$ is a Cauchy sequence in  $(\mfX, d_{\mfX})$, then for each integer $h \ge 1$ there is an integer $N_h \ge 1$ such that for all $n \ge N_h$ we have $f^{[h]}(T_n^\bullet) = f^{[h]}(T_{N_h}^\bullet)$. We may define a tree $T^\bullet \in \mfX$ such that $f^{[h]}(T^\bullet) = f^{[h]}(T_{N_h}^\bullet)$ for all $h \ge 1$. This tree satisfies $\lim_{n \to \infty} d_{\mfX}(T^\bullet, T_n^\bullet) =0$.
	
	The space is separable and $\mfX_\mathrm{f}$ is a countable dense subset: There are countably many finite plane trees. As $\mathfrak{G}$ is countable, each finite tree has a countable number of type configurations. Hence there are countably many finite $\mathfrak{G}$-type trees. Each has a finite number of locations for a marked vertex. Hence $\mfX_\mathrm{f}$ is countable. Moreover, for any integer $h \ge 0$ and any $T^\bullet \in \mfX \setminus \mfX_{\mathrm{f}}$ it holds that $f^{[h]}(T^\bullet) \in \mfX_\mathrm{f}$ and $d_{\mfX}(T^\bullet, f^{[h]}(T^\bullet)) = 2^{-h}$. We may force this distance to be arbitrarily small by taking $h$ sufficiently large.
\end{proof}

\subsection{Annealed and quenched convergence of random trees}

Suppose that $(\bm{T}_n)_{n \ge 1}$ is a sequence of random finite $\mathfrak{G}$-type plane trees. Let $\mathfrak{G}_0 \subset \mathfrak{G}$  be a non-empty subset so that 
\begin{align}
	\Prb{\sum_{i \in \mathfrak{G}_0} \#_i \bm{T}_n >0} \to 1
\end{align} as $n \to \infty$. We may select a vertex $v_n$ from $\bm{T}_n$ uniformly at random among all vertices of $\bm{T}_n$ whose type lies in $\mathfrak{G}_0$. That is, if vertices of such types are present, otherwise we set $v_n$ to some placeholder value.

\emph{Annealed} convergence of $(\bm{T}_n, v_n)$ refers to distributional convergence in the usual sense, that is
\begin{align}
	(\bm{T}_n, v_n) \convdis \bm{T}^\bullet
\end{align}
for some random variable  $\bm{T}^\bullet$ with values in $\mfX$. This is equivalent to requiring
\begin{align}
	\Pr{f^{[h]}(\bm{T}_n, v_n) = T^\bullet} \to \Pr{f^{[h]}(\bm{T}^\bullet) = T^\bullet}
\end{align}
for any integer $h \ge 0$ and any finite marked tree $T^\bullet \in \mfX_\mathrm{f}$. 

Let $\mathbb{M}_1(\mfX)$ denote the Polish space of Borel probability measures on $\mfX$. 
Given a finite $\mathfrak{G}$-type tree $T$ having at least one vertex  with type in $\mathfrak{G}_0$, we may  consider the probability measure $P_T \in \mathbb{M}_1(\mfX)$ that corresponds to the uniform measure on all pairs $(T,v)$, with $v$ a vertex of $T$ having type in $\mathfrak{G}_0$. Thus,
\begin{align}
	P_{\bm{T}_n} \eqdist \mfL( (\bm{T}_n, v_n) \mid \bm{T}_n).
\end{align}
We say $(\bm{T}_n,v_n)$ converges in the \emph{quenched} sense to $\bm{T}^\bullet$, if
\begin{align}
	P_{\bm{T}_n} \convp \mathfrak{L}(\bm{T}^\bullet).
\end{align}
This is equivalent to requiring that for any integer $h \ge 0$ and any finite marked tree $T^\bullet \in \mfX_\mathrm{f}$, with the marked vertex having type in $\mathfrak{G}_0$, it holds that 
\begin{align}
	\frac{ |\{ v \in \bm{T}_n \mid f^{[h]}(\bm{T}_n, v) = T^\bullet\}|}{ \sum_{i \in \mathfrak{G}_0} \#_i \bm{T}_n } \convp  \Pr{f^{[h]}(\bm{T}^\bullet) = T^\bullet}.
\end{align}

%\subsection{Multi-type Galton--Watson trees}

\section{Rerooted multi-type Galton--Watson trees}
\label{sec:rerooted}

Let $\mathfrak{G}$ denote a countable non-empty set of types.
Let $\bm{\sigma} = (\bm{\sigma}_i)_{i \in \mathfrak{G}}$ denote an independent family with $\bm{\sigma}_i$ a random word from $\langle \mathfrak{G} \rangle$ for each type $i \in \mathfrak{G}$. We let $\bm{\xi}_i$ denote the random vector corresponding to $\bm{\sigma}_i$. That is, $\bm{\xi}_i$ takes values in $\ndN_0^{( \mathfrak{G} )}$, and for each $j \in \mathfrak{G}$ the $j$-th coordinate of  $\bm{\xi}_i$ is equal to the number of occurrences of the letter $j$ in the random word $\bm{\sigma}_i$. Let $\bm{\xi} = (\bm{\xi}_i)_{i \in \mathfrak{G}}$. By abuse of notation, we call $\bm{\sigma}$ an \emph{ordered $\mathfrak{G}$-type offspring distribution}, and $\bm{\xi}$ the associated \emph{(unordered) $\mathfrak{G}$-type offspring distribution}.\label{de:offspring} 

A \emph{$\mathfrak{G}$-type branching process} with offspring distribution $\bm{\sigma}$ starts with a single individual whose type is determined by an independent random variable $\alpha$ with values in $\mathfrak{G}$. Any individual of type $i \in \mathfrak{G}$ receives offspring according to an independent copy of $\bm{\sigma}_i$, with the $j$-th coordinate of the copy corresponding to the number of children with type $j$. 

\label{de:bmt} We let $\bm{T}$ denote the $\mathfrak{G}$-type plane tree that encodes the genealogical structure of the $\mathfrak{G}$-type branching process, and say $\bm{T}$ is a \emph{$\bm{\sigma}$-Galton--Watson tree}.

We will always  make the assumption that
\begin{align}
	\label{eq:relevance}
	\Pr{\#_i \bm{T} \ne 0} >0 \qquad \text{for all $i \in \mathfrak{G}$}.
\end{align}
We may do so without loss of generality, since we may  shrink the set $\mathfrak{G}$ to exclude irrelevant types. Furthermore, we only consider offspring distributions  for which
\begin{align}
	\label{eq:fin}
	\text{$\bm{T}$ is almost surely finite}.
\end{align}
See for example~\cite[Thm. 2 on page 186]{MR0373040} for conditions ensuring this.

\subsection{Fringe subtrees}

Let us fix a type $\kappa \in \mathfrak{G}$. We let \label{de:ta}$\bm{T}^{\alpha}$ denote a random multi-type tree defined similarly to~$\bm{T}$ (with the root type determined according to an $\alpha$-distributed choice), only that non-root vertices of type $\kappa$ receive no offspring.  This way, $\bm{T}$ may be generated by starting with $\bm{T}^{\alpha}$ and inserting at each non-root vertex of type $\kappa$ an independent copy of $\bm{T}$ conditioned on having root type $\kappa$.  We let \label{de:tk}$\bm{T}^{\kappa}$  be defined analogously to $\bm{T}^\alpha$, but we start with a root having a fixed type~$\kappa$. Moreover, we let $\bm{T}^\kappa_1, \bm{T}^\kappa_2, \ldots$ denote independent copies of $\bm{T}^{\kappa}$. Assumption~\eqref{eq:fin} entails that
\begin{align}
	\Pr{\#_\kappa \bm{T} < \infty} = 1.
\end{align}
It follows by the standard depth-first-search exploration that the multi-type Galton--Watson tree $\bm{T}$ may be constructed from a  sequence
\begin{align}
	\label{eq:ttasdf}
	(\bm{T}^\alpha, \bm{T}^\kappa_1, \ldots, \bm{T}^\kappa_L)
\end{align}
with $L\ge 0$ the smallest non-negative integer for which
\begin{align}
	\label{eq:red2mono}
	\#_\kappa \bm{T}^\alpha - \one_{\alpha=\kappa} + \sum_{i=1}^L \left( \#_\kappa \bm{T}^\kappa_i -1 \right) = L.
\end{align}
In particular, 
\begin{align}
	\label{eq:oins}
	\#_\kappa \bm{T} = L + \one_{\alpha =\kappa}.
\end{align} 

Let	 $T$ be a finite $\mathfrak{G}$-type tree whose root has type $\kappa$. We may decompose $T$ in the same way, so that it corresponds to a sequence of trees $(T_1, \ldots, T_k)$. We let $N_T(\cdot)$\label{de:nt} denote a function that takes a finite $\mathfrak{G}$-type tree as input and returns the number of occurrences of $T$ as a fringe subtree.  We let $\psi(\cdot)$ denote a function that takes as input a finite ordered sequence of $\mathfrak{G}$-type trees and returns as output the number of occurrences of $(T_1, \ldots, T_k)$ as a consecutive substring. Generating $\bm{T}$ from the sequence in~\eqref{eq:ttasdf}, we may write
\begin{align}
	\label{eq:nt}
	N_T(\bm{T}) = \psi(\bm{T}^\alpha, \bm{T}^\kappa_1, \ldots, \bm{T}^\kappa_L).
\end{align}
Note that occurrences of $(T_1, \ldots, T_k)$ may not overlap: since $(T_1, \ldots, T_k)$  corresponds to a tree, no proper initial segment $(T_1, \ldots, T_i)$ with $1 \le i < k$ may be equal to a a proper tail segment $(T_{k-i+1}, \ldots, T_k)$.  Hence changing one coordinate of the input of the function $\psi$ changes its value by at most $1$.\footnote{Here is a helpful analogy: Suppose we count how often the string ``and'' occurs in a book as consecutive substring. Such occurrences cannot overlap, because no proper initial segment of the string is equal to a tail segment. Hence if we change one letter of the book, the total number of occurrences may increase by $1$, decrease by $1$, or stay the same.} %Conditionally on $L$, the coordinates of $(\bm{T}^\alpha, \bm{T}^\kappa_1, \ldots, \bm{T}^\kappa_L)$ are independent, hence 
It follows by McDiarmid's inequality  that for any integer $\ell \ge 1$
\begin{align}
	\label{eq:AA}
	\Prb{\left|\psi(\bm{T}^\alpha, \bm{T}^\kappa_1, \ldots, \bm{T}^\kappa_\ell) - \Exb{\psi(\bm{T}^\alpha, \bm{T}^\kappa_1, \ldots, \bm{T}^\kappa_\ell) }  \right| \ge x  }  \le 2 \exp\left( -\frac{2 x^2}{\ell+1}\right). 
\end{align}
Moreover, 
\begin{align}
	\label{eq:BB}
	\Exb{\psi(\bm{T}^\alpha, \bm{T}^\kappa_1, \ldots, \bm{T}^\kappa_\ell) }  &= O(1) + \ell \prod_{s=1}^k \Pr{\bm{T}^\kappa= T_s},
\end{align}
with the $O(1)$ term having a deterministic absolute bound that depends only on~$T$. Letting $\bm{T}(\kappa)$\label{de:tkappa} denote a $\bm{\sigma}$-Galton--Watson tree started at a vertex with type $\kappa$, it holds that 
\begin{align}
	\label{eq:zwoa}
	\prod_{s=1}^k \Pr{\bm{T}^\kappa= T_s} = \Pr{\bm{T}(\kappa) = T}.
\end{align}

\begin{lemma}
	\label{le:multfringe}
	We consider the conditioned multi-type Galton--Watson tree \label{de:tn} \begin{align}
		\label{eq:deftn}
		\bm{T}_n := (\bm{T} \mid \cE_n),
	\end{align} for some family of events $\cE_n$ satisfying $\Pr{\cE_n}>0$ for all $n$.  Suppose that there is a deterministic sequence $s_n\to \infty$ satisfying for all $\epsilon>0$ 
	\begin{align}
		\label{eq:boundassumption}
		\exp(-\epsilon s_n) = o( \Pr{\cE_n}) \qquad \text{and} \qquad \Pr{\#_\kappa \bm{T}_n \ge s_n } \to 1.
	\end{align}
	Then  for any finite $\mathfrak{G}$-type tree $T$ with root type $\kappa$
	\begin{align}
		\label{eq:as1}
		\frac{N_T(\bm{T}_n)}{\#_\kappa \bm{T}_n} \convp \Pr{\bm{T}(\kappa) = T}.
	\end{align}
	Let $f(\bm{T}_n, v_n^\kappa)$ denote the fringe subtree in $\bm{T}_n$ encountered at a uniformly selected type $\kappa$ vertex~$v_n^\kappa$ of $\bm{T}_n$. Then
	\begin{align}
		\label{eq:as2}
		\mfL(f(\bm{T}_n, v_n^\kappa) \mid \bm{T}_n) \convp \mfL(\bm{T}(\kappa)).
	\end{align}
\end{lemma}

Here we interpret the conditional law $\mfL(f(\bm{T}_n, v_n^\kappa) \mid \bm{T}_n)$ as a random Borel probability measure on the countable space of finite $\mathfrak{G}$-type plane trees equipped with the discrete topology.

\begin{proof}[Proof of Lemma~\ref{le:multfringe}]
	Using Equations~\eqref{eq:oins} and \eqref{eq:nt}, we obtain for any $\epsilon>0$
	\begin{align*}
		&\Prb{ \left| \frac{N_T(\bm{T}_n)}{\#_\kappa \bm{T}_n} - \Pr{\bm{T}(\kappa) = T} \right| \ge \epsilon} \\ 
		&\le  \Pr{L < s_n -1 \mid \cE_n}  \\&\quad+ \Pr{\cE_n}^{-1}\sum_{\ell \ge s_n-1} 	\Prb{ \left| \frac{N_T(\bm{T})}{\#_\kappa \bm{T}} - \Pr{\bm{T}(\kappa) = T} \right| \ge \epsilon, L = \ell} \\
		&\le \Pr{\#_\kappa \bm{T}_n < s_n }  \\&\quad + \Pr{\cE_n}^{-1}\sum_{\ell \ge s_n-1} 	\Prb{ \left| \frac{\psi(\bm{T}^\alpha, \bm{T}^\kappa_1, \ldots, \bm{T}^\kappa_\ell)}{\ell + \one_{\alpha = \kappa}} - \Pr{\bm{T}(\kappa) = T} \right| \ge \epsilon}.
	\end{align*}
	Combining Equations~\eqref{eq:AA}, \eqref{eq:BB}, and \eqref{eq:zwoa}, it follows that there is a constant $C>0$ that only depends on $T$, such that for $\ell \ge s_n - 1$ and $n$ large enough (so that $\epsilon \ell >C$) 
	\begin{align*}
		&\Prb{ \left| \frac{\psi(\bm{T}^\alpha, \bm{T}^\kappa_1, \ldots, \bm{T}^\kappa_\ell)}{\ell + \one_{\alpha = \kappa}} - \Pr{\bm{T}(\kappa) = T} \right| \ge \epsilon} \\
		&\le \Prb{ \left| \psi(\bm{T}^\alpha, \bm{T}^\kappa_1, \ldots, \bm{T}^\kappa_\ell) - \Exb{\psi(\bm{T}^\alpha, \bm{T}^\kappa_1, \ldots, \bm{T}^\kappa_\ell) } \right| \ge \epsilon\ell - C} \\
		&\le 2 \exp\left( -\frac{2 (\epsilon\ell - C)^2}{\ell+1}\right).
	\end{align*}
	
	Summing up, we obtain
	\begin{multline}
		\label{eq:lllbound}
		\Prb{ \left| \frac{N_T(\bm{T}_n)}{\#_\kappa \bm{T}_n} - \Pr{\bm{T}(\kappa) = T} \right| \ge \epsilon} \\ \le \Pr{\#_\kappa \bm{T}_n < s_n }  + \Pr{\cE_n}^{-1} O(\exp(-2 \epsilon^2 s_n)).
	\end{multline}
	Our assumptions in~\eqref{eq:boundassumption} ensure that the upper bound in Inequality~\eqref{eq:lllbound} tends to zero as $n$ becomes large. Hence this verifies~\eqref{eq:as1}.  
	
	We have verified~\eqref{eq:as1} for arbitrary finite $T$, and Assumptions~\eqref{eq:relevance} and \eqref{eq:fin} ensure that $\bm{T}(\kappa)$ is almost surely finite. Hence the convergence of random probability measures in $\eqref{eq:as2}$ follows.
\end{proof}

Of course, periodicities may come into play, and it is sensible to also consider  events $\cE_n$ for which $\Pr{\cE_n}>0$ only when $n$ is part of some infinite subset of $\ndN_0$. Nothing changes in our arguments as long as we restrict $n$ to that subset when taking limits. 

Lemma~\ref{le:multfringe} is rather general. Under stronger assumptions, its proof yields stronger results:

\begin{remark}
	We may apply Lemma~\ref{le:multfringe} if
	\begin{align}
		\Pr{\cE_n} = \exp(o(n)),
	\end{align} and if there is  a concentration constant $c(\kappa)>0$ such that $\Pr{\#_\kappa \bm{T}_n  \notin (1 \pm \epsilon)c(\kappa)n}$ tends to zero exponentially fast for any fixed $\epsilon>0$. Using~\eqref{eq:lllbound} it follows that    
	\begin{align}
		\Prb{ \left| \frac{N_T(\bm{T}_n)}{c(\kappa) n} - \Pr{\bm{T}(\kappa) = T}  \right| \ge \epsilon} = O(\gamma(\epsilon)^n)
	\end{align}  for some constant $0< \gamma(\epsilon) <1$. Hence by the Borel--Cantelli criterium 
	\begin{align}
		\label{eq:almostsure1}
		\frac{N_T(\bm{T}_n)}{c(\kappa) n} \convas \Pr{\bm{T}(\kappa) = T},
	\end{align}
	yielding
	\begin{align}
		\label{eq:almostsure2}
		\mfL(f(\bm{T}_n, v_n^\kappa) \mid \bm{T}_n) \convas \mfL(\bm{T}(\kappa)).
	\end{align}
\end{remark}

\subsection{Multi-type sin-trees with a fixed root type}

We are going to define a random infinite but locally finite multi-type tree $\hat{\bm{T}}(\kappa)$ having an infinite spine that growth backwards.

Note that $\#_\kappa \bm{T}(\kappa)$ is distributed like the population of a monotype Galton--Watson tree with branching mechanism $\#_\kappa \bm{T}^\kappa-1$. Assumptions~\eqref{eq:relevance} and ~\eqref{eq:fin} entail that
\begin{align}
	\text{$\bm{T}(\kappa)$ is almost surely finite}.
\end{align}
This means 
\begin{align}
	\Pr{\#_\kappa \bm{T}^\kappa = 1} > 0, \qquad \Pr{\#_\kappa \bm{T}^\kappa = 2} < 1, \qquad \text{and} \qquad \Ex{\#_\kappa \bm{T}^\kappa } \le 2.
\end{align}

The construction requires us to make the assumption that the $\kappa$-branching mechanism is actually critical. That is, we assume in the following that
\begin{align}
	\label{eq:critassumption}
	\Ex{\#_\kappa \bm{T}^\kappa} = 2.
\end{align}
This allows us to define the $\kappa$-biased version $\hat{\bm{T}}^\kappa$\label{de:thk} of $\bm{T}^\kappa$ with distribution given by
\begin{align}
	\label{eq:sbstep}
	\Pr{\hat{\bm{T}}^\kappa = (T^\kappa, u)} = \Pr{\bm{T}^\kappa = T^\kappa}
\end{align}
for any pair $(T^\kappa, u)$ of a  $\mathfrak{G}$-type tree $T^\kappa$ (with the root having type $\kappa$ and all non-root vertices of type $\kappa$ having no offspring) and a vertex $u$ of $T^\kappa$ that is a non-root vertex of type $\kappa$. Note that Assumption~\eqref{eq:critassumption} is really required in order for this to be a probability distribution.

We are now ready to construct the random infinite $\mathfrak{G}$-type tree $\hat{\bm{T}}(\kappa)$\label{de:htk}.
We start with a vertex $u_0$ of type $\kappa$ that is declared the start of the spine and becomes the root of an independent copy of $\bm{T}(\kappa)$. The ancestor $u_1$ of $u_0$ becomes the root of an independent copy of $\hat{\bm{T}}^\kappa$. We then glue the marked vertex to $u_0$, and all non-marked leaves of type $\kappa$ become roots of independent copies of $\bm{T}(\kappa)$. We proceed in this way with an ancestor $u_2$ of $u_1$ and so on, yielding an infinite backwards growing spine $u_0, u_1, \ldots$. That is, the tree $\hat{\bm{T}}(\kappa)$ obtained in this way has a marked vertex $u_0$ with a countably infinite number of ancestors. This  constitutes the multi-type analogue of Aldous' invariant sin-tree constructed in~\cite{MR1102319} for critical monotype Galton--Watson trees. Here the abbreviation \emph{sin} stands for \emph{single infinite path}.

Given an integer $h \ge 0$, we let $f^{\kappa, [h]}(\cdot, \cdot)$ denote a  function that takes as input a $\mathfrak{G}$-type  tree  $T_1$ together with one of its vertices $v_1$, and returns the fringe subtree of $T_1$ at the $h$-th ancestor of type $\kappa$ of $v_1$ together with the location $v_1$ within it. That is, it produces a marked tree where the root has type $\kappa$ and where the path from the root to the marked vertex contains precisely $h+1$ vertices of type $\kappa$. If no such ancestor exists, the function returns $(T_1, v_1)$ together with the information that an overflow occurred. It is immediate that
\begin{align}
	f^{\kappa, [0]}(\hat{\bm{T}}(\kappa)) \eqdist \bm{T}(\kappa)
\end{align}
and, in general, $f^{\kappa, [h]}(\hat{\bm{T}}(\kappa))$ follows the distribution of $\bm{T}(\kappa)$ biased on the number vertices with type $\kappa$ whose joining path with the root contains precisely $h+1$ vertices that also have type $\kappa$. That is, if $T$ is a $\mathfrak{G}$-type tree whose root has type $\kappa$ and if $u$ is  a vertex of $T$ of type $\kappa$ such that the joining path with the root contains precisely $h+1$ vertices of type $\kappa$, then
\begin{align}
	\label{eq:sbtr}
	\Pr{f^{\kappa, [h]}(\hat{\bm{T}}(\kappa))= (T,u)} = \Pr{\bm{T}(\kappa) = T}.
\end{align}

%We consider the collection $\mathfrak{T}^{\kappa, \bullet}$ of all  $\mathfrak{G}$-type trees $T^\bullet$ having a marked vertex so that $f^{\kappa, [h]}(T^\bullet)$ is finite for all $h \ge 0$. We may endow $\mathfrak{T}^{\kappa, \bullet}$ with a topology such that convergence (deterministic or  in distribution) of a (deterministic or random) sequence $(X_n)_{n \ge 1}$ in $\mathfrak{T}^{\kappa, \bullet}$ is equivalent to convergence of the projections $f^{\kappa, [h]}(X_n)$, $h \ge 0$. This is analogous to the monotype case studied in~\cite{rerootedstufler2019}.

Let $T$ denote a finite $\mathfrak{G}$-type tree and let $v$ be a vertex of $T$, such that there are $h+1$ vertices of type $\kappa$ on the path from $v$ to the root of $T$. Given a $\mathfrak{G}$-type tree $T_1$, we let  $N_{(T,v)}(T_1) \in \ndN_0 \cup \{\infty\}$ denote the number of vertices in $T_1$ with $f^{\kappa, [h]}(T_1,u) = (T,v)$ (without any overflow). It is obvious that the functions $N_{(T,v)}(\cdot)$ and $N_T(\cdot)$ always return the same number, hence
\begin{align}
	\label{eq:Neq}
	N_{(T,v)}(\cdot) = N_T(\cdot).
\end{align}
Together with Lemma~\ref{le:multfringe} and Equation~\eqref{eq:sbtr} we obtain:

\begin{theorem}
	\label{te:convfringe}
	Let $\bm{T}_n$ be as in Equation~\eqref{eq:deftn}.
	Suppose that $\Ex{\#_\kappa \bm{T}^\kappa} = 2$ and that Assumption~\eqref{eq:boundassumption} holds. Let $v_n^\kappa$  denote a uniformly selected  type $\kappa$ vertex of $\bm{T}_n$.  Then
	\begin{align}
		\mfL( (\bm{T}_n, v_n^\kappa) \mid \bm{T}_n) \convp \mathfrak{L}(\hat{\bm{T}}(\kappa)).
	\end{align} 
	That is, given  a finite $\mathfrak{G}$-type tree $T$ with root type $\kappa$ and a type $\kappa$ vertex $v$ of $T$, 
	\begin{align}
		\label{eq:as221}
		\frac{N_{(T,v)}(\bm{T}_n)}{\#_\kappa \bm{T}_n} \convp \Pr{f^{\kappa, [h]}(\hat{\bm{T}}(\kappa)) = (T,v)},
	\end{align}
	with $h+1$ denoting the number of vertices of type $\kappa$ on the path from the root to $v$ in $T$.
\end{theorem}

\subsection{Non-recurring types along the spine}

Theorem~\ref{te:convfringe} is a local convergence result for the vicinity of a uniformly selected vertex of type~$\kappa$ satisfying the criticality constraint  \eqref{eq:critassumption}. The limit tree $\hat{\bm{T}}(\kappa)$ has the property, that its spine has an infinite number of vertices of type $\kappa$. We are going to prove a criterion that also encompasses types that occur only a stochastically bounded number of times along the spine of the limit.

Suppose that we are given a type $\gamma \in \mathfrak{G} \setminus \{\kappa\}$ satisfying
\begin{align}
	\label{eq:biasgamma}
	0 < \Ex{\#_\gamma \bm{T}^\kappa} < \infty.
\end{align}
This allows us to define the size-biased version $\hat{\bm{T}}^{\kappa,\gamma}$\label{de:tkg} given by
\begin{align}
	\label{eq:biaskappagamma}
	\Pr{\hat{\bm{T}}^{\kappa,\gamma} = (T,u)} = \frac{\Pr{\bm{T}^\kappa = T}}{ \Ex{\#_\gamma \bm{T}^\kappa}}
\end{align}
for any finite $\mathfrak{G}$-type tree $T$ (with root type $\kappa$, and all non-root vertices of type $\kappa$ having no offspring) and any vertex $u$ of $T$ of type $\gamma$. We define $\hat{\bm{T}}(\kappa, \gamma)$\label{de:tkg2} like $\hat{\bm{T}}(\kappa)$, only with a single local modification: instead of letting $u_0$ become the root of an independent copy of $\bm{T}(\kappa)$, we let it become the root of an independent copy of $\hat{\bm{T}}^{\kappa,\gamma}$ (with the marked vertex becoming the marked vertex of $\hat{\bm{T}}(\kappa, \gamma)$), and then make each of the type $\kappa$ leaves of this structure the root of an independent copy of $\bm{T}(\kappa)$. Note that $\hat{\bm{T}}(\kappa, \gamma)$ may or may not have an infinite number of vertices of type $\gamma$ on the spine, depending on whether $\bm{T}^{\kappa}$ has with positive probability a vertex of type $\gamma$ on the path from the root of $\bm{T}^\kappa$ to some leaf of $\bm{T}^{\kappa}$ with type $\kappa$.

The representation~\eqref{eq:ttasdf} entails that, since $\gamma \ne \kappa$, 
\begin{align}
	\label{eq:typegamma}
	\#_\gamma \bm{T} = \#_\gamma \bm{T}^\alpha + \sum_{i=1}^{\#_\kappa \bm{T}-\one_{\alpha = \kappa}}  \#_\gamma \bm{T}^\kappa_i.
\end{align}
Hence if the events $\cE_n$ are reasonably well behaved,  Assumption~\eqref{eq:biasgamma}  and Equation~\eqref{eq:typegamma} will ensure that $\#_\gamma \bm{T}$ concentrates around $\Ex{\#_\gamma \bm{T}^\kappa} \#_\kappa \bm{T}$. This motivates the following local convergence result for the vicinity of a typical vertex with type~$\gamma$:
\begin{theorem}
	\label{te:convfringe2}
	Let $\bm{T}_n$ be as in Equation~\eqref{eq:deftn}.
	Suppose that $\Ex{\#_\kappa \bm{T}^\kappa} = 2$ and that Assumption~\eqref{eq:boundassumption} holds. Furthermore, assume that $0 < \Ex{\#_\gamma \bm{T}^\kappa}< \infty$. Then for any finite $\mathfrak{G}$-type tree $T$ with root type $\kappa$ and a type $\gamma$ vertex $v$ of $T$  it holds that
	\begin{align}
		\label{eq:as221444}
		\frac{N_{(T,v)}(\bm{T}_n)}{\Ex{\#_\gamma \bm{T}^\kappa} \#_\kappa \bm{T}_n} \convp \Pr{f^{\kappa, [h]}(\hat{\bm{T}}(\kappa, \gamma)) = (T,v)}
	\end{align}
	with $h+1$ denoting the number of vertices of type $\kappa$ on the path from the root to $v$ in $T$.
	Let $v_n^\gamma$  denote a uniformly selected  type $\gamma$ vertex of $\bm{T}_n$.  If
	\begin{align}
		\label{eq:dd}
		\frac{\#_\gamma \bm{T}_n}{\#_\kappa \bm{T}_n} \convp \Ex{\#_\gamma \bm{T}^\kappa},
	\end{align}
	then
	\begin{align}
		\label{eq:final}
		\mfL( (\bm{T}_n, v_n^\gamma) \mid \bm{T}_n) \convp \mathfrak{L}(\hat{\bm{T}}(\kappa, \gamma)).
	\end{align} 	
\end{theorem}
\begin{proof}
	Equation~\eqref{eq:as221444} readily follows from \eqref{eq:as1}, \eqref{eq:Neq}, and the definition of $\hat{\bm{T}}(\kappa, \gamma)$. Having Equation~\eqref{eq:as221444}  and Assumption~\eqref{eq:dd} at hand, \eqref{eq:final} readily follows by Slutsky's theorem.
\end{proof}

\begin{remark}
	\label{re:fu}
	Note that Theorem~\ref{te:convfringe2} also applies to the case $\gamma = \kappa$ if we replace $\Ex{\#_\gamma \bm{T}^\kappa}$ by~$1$ in the definition of $\hat{\bm{T}}(\kappa, \gamma)$ and in Assumption~\eqref{eq:dd}.
\end{remark}

\begin{proposition}
	\label{pro:strengthened}
	Theorem~\ref{te:convfringe2} still holds if we relax Equation~\eqref{eq:dd} to 
	\begin{align}
		\label{eq:relaxmate}
		\frac{\#_\gamma \bm{T}_n}{\#_\kappa \bm{T}_n} \le  \Ex{\#_\gamma \bm{T}^\kappa} + o_p(1).
	\end{align}
\end{proposition}
\begin{proof}
	Let $(T,v)$ be an arbitrary pointed tree where $T$ has root type $\kappa$ and $v$ has type $\gamma$ and height $h\ge 0$. If $\Pr{f^{\kappa, [h]}(\hat{\bm{T}}(\kappa, \gamma)) = (T,v)}=0$, then $N_{(T,v)}(\bm{T}_n)=0$ for all $n$, hence it suffices to verify 
	\begin{align}
		\label{eq:tos}
		\frac{N_{(T,v)}(\bm{T}_n)}{\#_\gamma \bm{T}_n} \convp \Pr{f^{\kappa, [h]}(\hat{\bm{T}}(\kappa, \gamma)) = (T,v)}
	\end{align}
	for the cases where $ \Pr{f^{\kappa, [h]}(\hat{\bm{T}}(\kappa, \gamma)) = (T,v)}>0$. Let $\Omega$ denote the collection of all such $(T,v)$. Equations~\eqref{eq:relaxmate} and \eqref{eq:as221444} imply that 
	\begin{align}
		\frac{N_{(T,v)}(\bm{T}_n)}{\#_\gamma \bm{T}_n} \ge \Pr{f^{\kappa, [h]}(\hat{\bm{T}}(\kappa, \gamma)) = (T,v)} - o_p(1).
	\end{align}
	That is, for any $\epsilon>0$ it holds with high probability that 
	\begin{align}
		\frac{N_{(T,v)}(\bm{T}_n)}{\#_\gamma \bm{T}_n} \ge \Pr{f^{\kappa, [h]}(\hat{\bm{T}}(\kappa, \gamma)) = (T,v)}(1 - \epsilon).
	\end{align}
	Suppose that there exists $(T_0, v_0) \in \Omega$ and $\epsilon_0, \delta_0>0$ and a subsequence $(n_k)_k$ with
	\[
	\Prb{\frac{N_{(T_0,v_0)}(\bm{T}_n)}{\#_\gamma \bm{T}_n} \ge \Pr{f^{\kappa, [h]}(\hat{\bm{T}}(\kappa, \gamma)) = (T_0,v_0)}(1 + \epsilon_0)} > \delta_0
	\]
	for all $n$ in that subsequence. Letting $\Omega' \subset \Omega$ denote a  finite subset with $\Pr{f^{\kappa, [h]}(\hat{\bm{T}}(\kappa, \gamma)) \in \Omega')}>1-\epsilon$ and $(T_0,v_0) \in \Omega'$, it follows that with probability at least $\delta_0 + o(1)$ for all $n$ in that subsequence
	\[
	1 \ge \sum_{(T,v) \in \Omega'} 	\frac{N_{(T,v)}(\bm{T}_n)}{\#_\gamma \bm{T}_n} \ge (1-\epsilon)^2 + (\epsilon + \epsilon_0)\Pr{f^{\kappa, [h]}(\hat{\bm{T}}(\kappa, \gamma)) = (T_0,v_0)}.
	\]
	Taking $\epsilon$ small enough, this yields a contradiction. This verifies \eqref{eq:tos} for all $(T,v) \in \Omega$ and completes the proof.
\end{proof}

\subsection{Mixtures of types}
Lemma~\ref{te:convfringe2} gives  a criterium for local convergence describing the vicinity of random vertices with a fixed type. We aim to prove limits for the conditioned tree $\bm{T}_n$ describing the local structure near a specified vertex that may have different types.

Given an integer $h \ge 0$, we let $f^{[h]}(\cdot, \cdot)$ denote a  function that takes as input a $\mathfrak{G}$-type tree together with one of its vertices, and returns the fringe subtree at the $h$-th ancestor of that vertex together with the location of the vertex within it (yielding a marked tree where the root and the marked vertex have distance $h$ from each other). If no such ancestor exists, the function returns the marked tree together with the information, that an overflow occurred.

%Analogously as for $\mathfrak{T}^{\kappa, \bullet}$, we may consider the collection $\mathfrak{T}^{\bullet}$ of all  $\mathfrak{G}$-type trees $T^\bullet$ having a marked vertex such that $f^{[h]}(T^\bullet)$ is finite for all $h \ge 0$. Again we may  endow $\mathfrak{T}^\bullet$ with a topology such that (deterministic or distributional) convergence is equivalent to convergence of the  projections $(f^{[h]}(\cdot))_{h \ge 0}$. 

The following observation follows directly from Theorem~\ref{te:convfringe2} and Remark~\ref{re:fu}.

\begin{theorem}
	\label{te:convfringe3} 		Let $\bm{T}_n$ be as in Equation~\eqref{eq:deftn}.
	Suppose that $\Ex{\#_\kappa \bm{T}^\kappa} = 2$ and that Assumption~\eqref{eq:boundassumption} holds. Let $\mathfrak{G}_0 \subset \mathfrak{G}$ be a non-empty subset such that 
	\begin{align}
		\label{eq:fo}
		\frac{\#_\gamma \bm{T}_n}{\#_\kappa \bm{T}_n} \convp \Ex{\#_\gamma \bm{T}^\kappa} \in ]0, \infty[
	\end{align}
	for each $\gamma \in \mathfrak{G}_0\setminus\{\kappa\}$. Let $v_n$ be a vertex of $\bm{T}_n$ drawn uniformly at random from all vertices with type in $\mathfrak{G}_0$. We set $p(\gamma) = \Ex{\#_\gamma \bm{T}^\kappa}$ for $\gamma \in \mathfrak{G}_0\setminus\{\kappa\}$, and $p(\kappa) =1$ in case $\kappa \in \mathfrak{G}_0$. If
	\begin{align}
		\label{eq:add}
		\frac{1}{\#_\kappa \bm{T}_n} \sum_{\gamma \in \mathfrak{G}_0} \#_\gamma \bm{T}_n \convp \sum_{\gamma  \in \mathfrak{G}_0} p(\gamma) < \infty,
	\end{align}
	then 
	\begin{align}
		\label{eq:finalthm}
		\mfL( (\bm{T}_n, v_n) \mid \bm{T}_n) \convp \mathfrak{L}(\hat{\bm{T}}(\kappa, \eta))
	\end{align} 	
	for an independent random type $\eta$ from $\mathfrak{G}_0$, with distribution given by
	\begin{align}
		\Pr{\eta= \gamma'} = p(\gamma') / \sum_{\gamma  \in \mathfrak{G}_0} p(\gamma), \qquad \gamma' \in \mathfrak{G}_0.
	\end{align}
\end{theorem}

Note that if $\mathfrak{G}_0$ is finite, then~\eqref{eq:add} readily follows from~\eqref{eq:fo}. We provide an application of Theorem~\ref{te:convfringe3} to critical sesqui-type trees in Section~\ref{sec:sesqu} and to critical irreducible Galton--Watson trees in Section~\ref{sec:regcrit}

Note that using Theorem~\ref{te:convfringe} we may still establish local convergence if the proportion of vertices of a certain type concentrates at different values as in Equation~\eqref{eq:fo}, but this works only for types that recur infinitely often on the spine in the limit.

\section{Applications}
\label{sec:applications}

We are going to illustrate the general results of the previous section by some examples.

\subsection{Lattices and the Gnedenko local limit theorem}

\label{sec:appendix}

Before we start, let us discuss a few relevant concepts. Given an integer $d \ge 1$, a \emph{lattice in  $\ndZ^d$} is a subset of the form \[
L_{\bm{a}, \bm{A}} = \{\bm{a} + \bm{A} \bm{x} \mid \bm{x} \in \ndZ^d\}
\]
for  $\bm{a} \in \ndZ^d$ and $\bm{A} \in \ndZ^{d \times d}$. For any $\bm{b} \in L_{\bm{a}, \bm{A}}$ it holds that $L_{\bm{a}, \bm{A}} = L_{\bm{b}, \bm{A}}$. In particular, if we shift a lattice by the negative of any of its elements, we obtain a $\ndZ$-linear submodule of $\ndZ^d$.  Note that the matrix $\bm{A}$ is not uniquely determined by the lattice. Any  matrix whose columns form a $\ndZ$-linear generating set of $\{\bm{A} \bm{x} \mid \bm{x} \in \ndZ^d\}$  will do.

Given a non-empty subset $\Omega \subset \ndZ^d$   there is smallest lattice in $\ndZ^d$ containing $\Omega$: We may select $\bm{a} \in \Omega$ and let $\Lambda$ denote the $\ndZ$-span of $\Omega - \bm{a}$. Any $\ndZ$-submodule of $\ndZ^d$ has rank at most $d$, hence there is at least one matrix $\bm{A} \in \ndZ^{d \times d}$ with $\Lambda = \{\bm{A} \bm{x} \mid \bm{x} \in \ndZ^d\}$. Hence $L_{\bm{a}, \bm{A}}$ is a lattice containing $\Omega$. Moreover, $L_{\bm{a}, \bm{A}}$ is a subset of any lattice containing $\Omega$: If $S = L_{\bm{b}, \bm{B}}$ is another lattice containing $\Omega$, then $\bm{a} \in S$ and consequently $S = L_{\bm{a}, \bm{B}}$. This entails that the $\ndZ$-module $S - \bm{a}$ must contain $\Omega - \bm{a}$, and therefore  $\Lambda$ is a submodule of $S - \bm{a}$. Hence $L_{\bm{a}, \bm{A}} \subset S$. 

We say a random vector $\bm{X}$ with values in an abelian group $F \simeq \ndZ^d$ is \emph{aperiodic}, if the smallest subgroup $F_0$ of $F$ that contains the support $\mathrm{supp}(\bm{X})$ satisfies $F_0 = F$. Note that this not really a restriction, since the structure theorem for finitely generated modules over a principal ideal domain ensures that $F_0 \simeq \ndZ^{d'}$ for some $0 \le d' \le d$. We say $\bm{X}$ (and the associated random walk with step distribution~$\bm{X}$) is \emph{strongly aperiodic}, if the smallest semi-group (a subset closed under addition that contains $\bm{0}$)  $F_1$ of $F$ that contains $\mathrm{supp}(\bm{X})$ satisfies $F_1 = F$. This is an actual restriction. For example, if $\bm{X}$ is $2$-dimensional with support $\{ (1,0), (0,1), (1,2) \}$, then it is aperiodic, but not strongly aperiodic (although the support is not even contained on any straight line). 

\begin{proposition}[{\cite[Prop. 1]{MR0279849}}]
	\label{prop:part}
	Let $\bm{X}$ be a random vector in $\ndZ^d$. Let  $\bm{a} + \bm{D} \ndZ^d$ be the smallest lattice  containing the support of $\bm{X}$. Suppose that  $\bm{a} \in \mathrm{supp}(\bm{X})$ and that $\bm{D}$ has full rank, so that $m := |\det \bm{D}|$ is a positive integer. Let $(\bm{X}_i)_{i \ge 1}$ denote independent copies of $\bm{X}$ and set
	\begin{align}
		\bm{S}_n := \sum_{i=1}^n \bm{X}_i.
	\end{align}
	Then:
	\begin{enumerate}
		%		\item The cosets $j \bm{a} + \bm{D} \ndZ^d$, $1 \le j \le m$, are mutually disjoint.
		\item The support of $\bm{S}_m$ generates $\bm{D} \ndZ^d$ as additive group.
		\item For all $k \ge 1$ and $1 \le j \le m$ it holds that $\Pr{\bm{S}_{km +j} \in j \bm{a} + \bm{D} \ndZ^d} = 1$.
	\end{enumerate}
\end{proposition}

The following is a strengthened and generalized multi-dimensional version of Gnedenko's local limit theorem, that applies to the case of lattice distributed random variables that are aperiodic but not necessarily strongly aperiodic.

\begin{proposition}[Strengthened local central limit theorem for lattice distributions]
	\label{pro:llt}
	Let $\bm{X}$, $m$, $\bm{a}$, $\bm{D}$ and $\bm{S}_n$ be as in Proposition~\ref{prop:part}. 
	Suppose that $\bm{X}$ has a finite  covariance matrix $\bm{\Sigma}$. Our assumptions imply that $\bm{\Sigma}$ is positive-definite. Let
	\[
	\varphi_{\bm{0}, \bm{\Sigma}}(\bm{y}) = \frac{1}{\sqrt{(2 \pi)^d \det \bm{\Sigma}}} \exp\left( - \frac{1}{2} \bm{y} \bm{\Sigma}^{-1} \bm{y} \right)
	\]
	be the  density of the normal distribution $\cN(\bm{0}, \bm{\Sigma})$. Set
	\[
	\bm{a}_n = n \Ex{\bm{X}} \qquad \text{and} \qquad B_n = \sqrt{n},
	\]
	so that $B_n^{-1}(\bm{S}_n - \bm{a}_n) \convd \cN(\bm{0}, \bm{\Sigma})$. Set 	\[
	R_n(\bm{x}) = \max\left(1, {  \| B_n^{-1} (\bm{x} - \bm{a}_n) \|_2^2} \right).
	\]
	Then for each integer $1 \le j \le m$ 
	\begin{align}
		\label{eq:tof}
		\lim_{\substack{n \to \infty \\ n \in j + m \ndZ}} \sup_{\bm{x} \in j\bm{a} + \bm{D} \ndZ^d} R_n(\bm{x})  |B_n^d  \Pr{\bm{S}_n = \bm{x}} - m \varphi_{\bm{0},\bm{\Sigma}}\left(B_n^{-1}(\bm{x} - \bm{a}_n) \right)| = 0.
	\end{align}
\end{proposition}
The one-dimensional local limit theorem by~\cite{MR0026275} was generalized by~\cite{rva} to the lattice case of strongly aperiodic multi-dimensional random walk. This was further generalized  by~\cite[Prop. 2]{MR0279849} to the setting considered here, where $\bm{X}$ is aperiodic but not necessarily strongly aperiodic. Furthermore, the version stated in Proposition~\ref{pro:llt} is strengthened by the factor $R_n(\bm{x}) $ in Equation~\eqref{eq:tof}.  This is stronger than the original statement when $\bm{x}$ deviates sufficiently from $\Ex{\bm{S}_n}$. Such a strengthening was obtained by~\cite[Statement P10, page 79]{MR0388547} for the strongly aperiodic case by modifying the proof. This strengthened version in the strongly aperiodic setting may be generalized to Proposition~\ref{pro:llt} analogously as in~\cite{MR0279849}.

\subsection{Sesqui-type trees}
\label{sec:sesqu}

We start with the simplest model of a non-monotype branching type process which already has interesting applications.
Consider a $2$-type branching tree $\bm{T}$ where only vertices of the first type are fertile and receive offspring according to an ordered branching mechanism $\bm{\sigma}_1$. We let $\bm{\xi}_1 = (\xi, \zeta)$ denote the corresponding vector counting the number of occurrences of type $1$ and type $2$ vertices in the random word $\bm{\sigma}_1$. Of course we only consider the case where the root of $\bm{T}$ is fertile. 

In order to avoid degenerate cases we assume that 
\begin{align}
	\Pr{\xi=0}>0 \qquad \text{and} \qquad \Pr{\xi\ge2}>0.
\end{align}
As the mono-type case is already well-understood, we additionally assume that the support 
\begin{align}
	\mathrm{supp}(\bm{\xi}_1) := \{ \bm{v} \in \ndZ^2 \mid \Pr{\bm{\xi}_1 = \bm{v}} >0\}
\end{align} of $\bm{\xi}_1$ is not contained on a straight line, that is
\begin{align}
	\mathrm{supp}(\bm{\xi}_1) \nsubseteq \bm{a} + \ndR \bm{v} \qquad \text{for all} \qquad \bm{a}, \bm{v} \in \ndR^2.
\end{align}
If the support of $\bm{\xi}_1$ is contained in a straight line  we can reduce   every question about $\bm{T}$ to the study of a mono-type $\xi$-Galton--Watson tree. Hence this is not much of a restriction. 

We consider the tree $\bm{T}_n$ obtained by conditioning $\bm{T}$ on having $n$ vertices in total.  We would like to establish a limit describing the vicinity of a uniformly at random selected vertex~$v_n$.  Due to possible periodicities, $n$ may need to be restricted so some infinite subset of the positive integers in order for this to make sense:

\begin{lemma}
	\label{le:lemn}
	Let $\mathrm{supp}(\#\bm{T}) \subset \ndN$ denote the support of the number of vertices of the tree $\bm{T}$. Setting
	\begin{align}
		\label{eq:ad}
		a := \min(\mathrm{supp}(\#\bm{T})) \qquad \text{and} \qquad D := \gcd\left(\mathrm{supp}(\#\bm{T}) - a\right),
	\end{align}
	it holds that $\mathrm{supp}(\#\bm{T}) \subset a + D \ndN$. Conversely, for all sufficiently large integers $n \in a + D \ndN$ 
	\begin{align}
		\Pr{\#\bm{T} = n} > 0.
	\end{align}
\end{lemma}
\begin{proof}
	It is clear by construction that $\mathrm{supp}(\#\bm{T}) \subset a + D \ndN$. The difficult and important part  is the converse direction. The generating series $Z(z) := \Ex{z^{\#\bm{T}}}$ satisfies the recursion
	\begin{align}
		\label{eq:recu}
		Z(z) = z f(Z(z), z)
	\end{align}
	with $f(x,y) := \Ex{x^{\xi} y^{\zeta}}$. We may rewrite this as
	\begin{align}
		Z(z) = \sum_{k \ge 0} E_k(z) Z(z)^k 
	\end{align}
	for power series $(E_k(z))_{k \ge 0}$ with non-negative coefficients uniquely determined by $
	zf(y,z) = \sum_{k \ge 0} E_k(z) y^k$.
	Let $F_k$ denote the collection of integers $i$ such that the coefficient of the monomial $z^i$ in $E_k(z)$ is positive.  We are going to use the commutative operation
	\[
	I + J := \{i + j \mid i \in I, j \in J\}
	\]
	for subsets $I, J \subset \ndZ$, and set $k + J := \{k\} + J$ for any integer $k$. Note that by definition $I + \emptyset = \emptyset$. It was shown by \cite[Lem. 25]{MR2240769} that
	\begin{align}
		\label{eq:BIG}
		D = \gcd \bigcup_{k \ge 0} S_k
	\end{align}
	for
	\begin{align}
		\label{eq:defSk}
		S_k := (k-1)a + F_k, \qquad k \ge 0.
	\end{align}
	Note that $S_k = \emptyset$ whenever $F_k = \emptyset$. Moreover, note that
	\begin{align}
		\label{eq:insup}
		\Pr{ (\xi, \zeta) = (k,b-1)} >0
	\end{align}
	for all $k \ge 0$ and $b \in F_k$. We are going to argue that
	\begin{align}
		\label{eq:toshow1}
		D = \gcd\left(S_1 \cup \bigcup_{k \ge 2} (S_k + S_0) \right).
	\end{align}
	It is clear that $D \mid S_k$ for all $k \ge 0$ by \eqref{eq:BIG}, so clearly $D$ divides the right-hand side of Equation~\eqref{eq:toshow1}. Conversely, let $r$ be a divisor of the right-hand side of this Equation. As $0 \in S_0$ it follows that $r \mid S_k$ for all $k \ge 1$. Moreover, we assumed that there exists $k \ge 2$ with $\Pr{\xi=k}>0$, implying $S_k \ne \emptyset$. As $r \mid S_k + S_0$ and $r \mid S_k \ne \emptyset$ it follows that $r \mid S_0$. Hence $r \mid S_k$ for all $k \ge 0$ and hence, by Equation~\eqref{eq:BIG}, $r \mid D$. This completes the verification of Equation~\eqref{eq:toshow1}.
	
	It follows by Schur's theorem, stated in~\cite[Thm. 3.24]{MR2172781}, that there exists $M \ge 1$ such that for all $m \ge M$ 
	\[
	mD \in \ndN_0 S_1 + \sum_{k \ge 2} \ndN_0 (S_k + S_0).
	\]
	That is, $mD$ may be expressed as a finite sum of terms of the form $\lambda t$ with $t \in S_1$, $\lambda \in \ndN$, and terms of the form $\mu(s + r)$ with $s \in S_k$ for some $k \ge 2$, $r \in S_0$, and $\mu \in \ndN$. 	 We are going to argue that  there is a tree $T$ with $a + Dm$ vertices that satisfies
	\begin{align}
		\label{eq:insupport}
		\Pr{\# \bm{T} = T} >0.
	\end{align}
	To this end, note that
	\begin{align}
		\label{eq:doh}
		a = 1 + \min\{i \ge 0 \mid \Pr{ (\xi,\zeta)=(0, i)} > 0 \} = \min F_0.
	\end{align}
	We start the growth construction of $T$ with a single root vertex that we declare as \emph{marked}. We iterate over the finitely many terms in the sum expression of $mD$, and each step takes as input a tree with a single marked leaf and outputs a bigger tree with a single marked leaf:
	\begin{enumerate}
		\item If the summand is of the form $\lambda t$ with $\lambda \in \ndN$ and $t \in S_1$, then the marked leaf receives $t-1$ type $2$  offspring vertices and a single type $1$ offspring vertex that becomes the new marked leaf. Note that the outdegree $(1,t-1)$ lies in the support of $\bm{\xi}_1$ by Equation~\eqref{eq:insup}. We do this precisely $\lambda$ many times. Hence in total we added $\lambda t$ vertices.
		\item If the summand is of the form $\mu(r + s)$ with $\mu \in \ndN$, $r \in S_k$ for some $k \ge 1$ and $s \in S_0$, we do the following. By Equation~\eqref{eq:defSk} there is $b \in F_k$ and $c \in F_0$ such that
		\begin{align*}
			r + s &= (k-1)a + b -a + c \\
			&= (k-2)(a-1) + (c-1) + k + (b-1).
		\end{align*}
		The marked leaf receives mixed offspring according to $(k, b-1)$. The first offspring of type $1$ becomes the new marked leaf. The second offspring of type $1$ receives offspring $(0, c-1)$. The remaining $k-2$ offspring vertices of type $1$ each receive offspring $(0,a-1)$. Note that by Equation~\eqref{eq:insup} all non-marked vertices have an outdegree that lies in the support of $\bm{\xi}_1$. We perform this precisely $\mu$ many times. Hence in total we added $\mu(r+s)$ vertices.
	\end{enumerate}
	After iterating over all summands we are left with a tree  having $1 + mD$ vertices (remember that we also have to count the root vertex) that has a marked leaf. All non-marked vertices have an outdegree that lies in the support of $\bm{\xi}_1$.  The marked vertex receives offspring $(0, a-1)$, resulting in a tree $T$ with $a + mD$ vertices that satisfies~\ref{eq:insupport}. As we may perform this construction for all $m \ge M$ this completes the proof.
\end{proof}

Setting $\kappa=1$, the tree $\bm{T}^\kappa$ consists of a root of type $1$ with offspring according to $\bm{\xi}_1$ and no further descendants. Hence in order to apply Theorem~\ref{te:convfringe3} we need to assume that 
\begin{align}
	\label{eq:asy}
	\Ex{\xi} = 1 \qquad \text{and}  \qquad 0<\Ex{\zeta}<\infty.
\end{align} 
It holds trivially that $\# \bm{T} \ge \#_1 \bm{T}$, and $\#_1 \bm{T}$ is distributed like the total progeny of a $\xi$-Galton--Watson process. As we assume that $\Ex{\xi}=1$, it follows that $\Pr{\#_1 \bm{T} = n} = \exp(o(n))$  (see \cite[Thm. 18.1]{MR2908619}) and consequently: \begin{align}
	\Pr{\# \bm{T} = n} = \exp(o(n)).
\end{align}
Hence the only prerequisite that we still need to check in order to apply Theorem~\ref{te:convfringe3} is that
\begin{align}
	\label{eq:proportion}
	\frac{ (\#_1 \bm{T} \mid \# \bm{T} = n )}{n} \convp \frac{1}{1 + \Ex{\zeta}}.
\end{align}
Indeed, if~\eqref{eq:proportion} is verified, then~\eqref{eq:boundassumption} holds for $s_n := cn$  for some fixed constant $0<c<1/(1 - \Ex{\zeta})$. Furthermore,~\eqref{eq:proportion} and $\#\bm{T} = \#_1\bm{T} + \#_2\bm{T}$ imply $	\frac{ \#_2 \bm{T}_n }{n} \convp \frac{\Ex{\zeta}}{1 + \Ex{\zeta}}$ and hence $\frac{\#_2 \bm{T}_n}{\#_1 \bm{T}_n} \convp \Ex{\zeta}$, thus verifying~\eqref{eq:fo} for $\gamma:=2$.

We verify Equation~\eqref{eq:proportion} in the case that $(\xi, \zeta)$ has a finite covariance matrix~$\bm{\Sigma}$. The following is an extension of~\cite[Lem. 22]{StEJC2018}, where  a combinatorial setting was considered that is related to the case where one additionally assumes that $a=1$ and that $(\xi, \zeta)$ has finite exponential moments.
\begin{lemma}
	\label{le:lembiv}
	Assume that $\Pr{\xi=0}>0$, $\Pr{\xi \ge 2}>0$, and that the support of $(\xi, \zeta)$ is not contained on a straight line.
	Suppose that $\Ex{\xi}=1$ and that $(\xi, \zeta)$ has a finite covariance matrix $\bm{\Sigma}$. Let $a, D \ge 1$ be defined as in Equation~\eqref{eq:ad}. 
	\begin{enumerate}
		\item There is a rank $2$ matrix $\bm{D} \in \ndZ^{2 \times 2}$ such that the support of $(\xi, \zeta)^\intercal$ is contained in the lattice $(0, a-1)^\intercal + \bm{D}\ndZ^2$ and in no proper sublattice. It holds that
		\begin{align}
			\label{eq:doo}
			\mathfrak{d}  := \frac{|\det \bm{D}|}{D} \in \ndN.
		\end{align}
		\item As $n \in a + D \ndZ$ tends to infinity, we have
		\begin{align}
			\label{eq:asymptoticT}
			\Pr{\# \bm{T} = n} \sim D \sqrt{\frac{1 + \Ex{\zeta}}{2 \pi \Va{\xi}}} n^{-3/2}.
		\end{align}
		\item Set
		\begin{align}
			\mu = \frac{1}{1 + \Ex{\zeta}} \qquad \text{and} \qquad \sigma^2 = \frac{\det \bm{\Sigma}}{\Va{\xi}(1 + \Ex{\zeta})^3}.
		\end{align}
		There is a sequence $(j_n)_n$ with values in $\{1, \ldots, \mathfrak{d}\}$ with the following property. As $n \in a + D \ndZ$ tends to infinity, 
		\begin{align}
			\label{eq:makeuniform}
			\sqrt{n} \Pr{\#_1\bm{T} = \ell \mid \#\bm{T} = n} = \frac{\mathfrak{d}}{\sigma \sqrt{2\pi}} \exp \left( -\frac{x^2}{2 \sigma^2}\right) + o(1)
		\end{align}
		uniformly for all $x$ satisfying 
		\begin{align}
			\ell := \mu n +x \sqrt{n} \in j_n + \mathfrak{d}\ndZ.
		\end{align}
		The probability in~\eqref{eq:makeuniform} equals zero when $\ell \notin j_n + \mathfrak{d}\ndZ$.
		Hence
		\begin{align}
			\label{eq:clt}
			\frac{(\#_1\bm{T} \mid \#\bm{T} = n) - n \mu}{\sqrt{n}} \convdis \cN(0, \sigma^2).
		\end{align}
	\end{enumerate}	
\end{lemma}
\begin{proof}
	We set
	\begin{align}
		\bm{X} := \colvec{2}{\xi -1}{\zeta+1}
	\end{align}
	and let 
	\begin{align}
		\bm{X}_i = \colvec{2}{\xi_i-1}{\zeta_i +1}, \qquad i \ge 1
	\end{align}
	denote independent copies of $\bm{X}$. We also set 
	\begin{align}
		\bm{S}_\ell = \sum_{i=1}^\ell \bm{X}_i \qquad \text{and} \qquad \bm{y}_n = \colvec{2}{-1}{n}.
	\end{align}
	Equations~\eqref{eq:red2mono} and~\eqref{eq:oins} tell us that for $\ell \ge 1$
	\begin{align}
		\label{eq:nlpre}
		\Pr{ \#\bm{T} = n, \#_1\bm{T} = \ell} &= \Prb{  \bm{S}_\ell = \bm{y}_n, \sum_{i=1}^k \xi_i \ge k \text{ for $0 < k < \ell$}},
	\end{align}
	with $(\xi_i,\zeta_i)_{i \ge 1}$ denoting independent copies of $(\xi, \zeta)$.
	By the cycle lemma (see for example~\cite{MR0138139}), this simplifies to
	\begin{align}
		\label{eq:nl}
		\Pr{ \#\bm{T} = n, \#_1\bm{T} = \ell} = \frac{1}{\ell}\Prb{ \bm{S}_\ell = \bm{y}_n}.
	\end{align}
	The support of $\bm{\xi}_1$  was assumed to be not contained in a straight line, hence the same goes for $(\xi-1, \zeta+1)$.  Hence  the covariance matrix $\bm{\Sigma}$ (of both $\bm{\xi}_1$ and the shifted version $(\xi-1, \zeta+1)$) is positive definite. Furthermore, there is a rank $2$ matrix  $\bm{D} \in \ndZ^2$ such that  the support of $(\xi-1, \zeta+1)^\intercal$ is contained in the lattice $\bm{a} + \bm{D}\ndZ^2$, with $\bm{a} := (-1,a)^\intercal$, and in no proper sublattice. (Recall that $(0, a-1)$ lies in the support of $\bm{\xi}_1$ by Equation~\eqref{eq:doh}.)  We set
	\begin{align}
		m := |\det \bm{D}|.
	\end{align} 
	Let $n \in a + D \ndZ$ be sufficiently large so that $\Pr{\#\bm{T} = n}>0$ by Lemma~\ref{le:lemn}. By Equation~\eqref{eq:nl} this means that for at least one $1 \le \ell \le n$ it holds that $\bm{y}_n$ lies in the support of $\bm{S}_\ell$. 
	By Proposition~\ref{prop:part} it follows that there is at least one integer $1 \le j_n \le m$ with
	\begin{align}
		\bm{y}_n \in j_n \bm{a} + \bm{D} \ndZ^2.
	\end{align}
	Note that for all $j,j' \in \ndZ$ the following statements are equivalent:
	\begin{enumerate}
		\item $(j \bm{a} + \bm{D} \ndZ^2) \cap (j' \bm{a} + \bm{D} \ndZ^2) \neq \emptyset$.
		\item $\bm{0} \in (j-j')\bm{a} + \bm{D} \ndZ^2$.
		\item $(j-j')\bm{a} \in \bm{D}\ndZ^2$.
		\item $j \bm{a} + \bm{D} \ndZ^2 = j' \bm{a} + \bm{D} \ndZ^2$.
	\end{enumerate}
	The set of all $j \in \ndZ$ with $j \bm{a} \in \bm{D} \ndZ^2$ is a subgroup of the integers. By Proposition~\ref{prop:part} it contains $m$. Hence it is generated by some integer $\mathfrak{d} \ge 1$ satisfying
	\begin{align}
		\mathfrak{d} | m.
	\end{align}
	We will postpone showing that we assume that $\mathfrak{d}$ may be chosen as in~\eqref{eq:doo}.  Moreover, we have 
	\begin{align}
		\bm{y}_n \in j \bm{a} + \bm{D} \ndZ^2 \qquad \text{if and only if} \qquad j \in j_n + \mathfrak{d} \ndZ.
	\end{align}
	Hence, from now on we may assume additionally that $1 \le j_n \le \mathfrak{d}$. By Equation~\eqref{eq:nl} it follows that any integer $\ell \ge 1$ with $\Pr{ \#\bm{T} = n, \#_1\bm{T} = \ell}>0$ must lie in the lattice $j_n + \mathfrak{d} \ndZ$. It is clear that not the entire lattice has this property. However, we will argue that this is true for all $\ell \in j_n + \mathfrak{d} \ndZ$ that concentrate in a $\sqrt{n}$ range around $\mu n$. Given $\ell \in j_n + \mathfrak{d} \ndZ$ with $\ell \ge 1$ we may write
	\begin{align}
		\ell = \mu n + x_\ell \sqrt{n}.
	\end{align}
	Let $M>0$ be a fixed constant. It holds uniformly for all $\ell$ with $|x_\ell| < M$ that
	\begin{align}
		\frac{1}{\sqrt{\ell}} \left( \bm{y}_n - \ell \Ex{\bm{X}} \right) = \colvec{2}{0}{-x_\ell (1 + \Ex{\zeta})^{3/2}} + o(1).
	\end{align}
	Hence
	\begin{align*}
		-\frac{1}{2\ell}&(\bm{y}_n - \ell \Exb{\bm{X}})^\intercal \bm{\Sigma}^{-1} (\bm{y}_n - \ell \Exb{\bm{X}}) \\
		&= -\frac{1}{2}  (0,-x_\ell (1 + \Ex{\zeta})^{3/2}) 
		\begin{pmatrix}
			\Va{\xi} & \mathrm{Cov}(\xi,\zeta)\\
			\mathrm{Cov}(\xi,\zeta) & \Va{\zeta}
		\end{pmatrix}^{-1}
		\colvec{2}{0}{-x_\ell (1 + \Ex{\zeta})^{3/2}} \\
		& \quad\, + o(1) \\
		&= - \frac{x_\ell^2 \Va{\xi} (1 + \Ex{\zeta})^3}{2 \det \bm{\Sigma} } + o(1).
	\end{align*}
	It follows from the multivariate local limit theorem in Proposition~\ref{pro:llt} that
	\begin{align*}
		\frac{1}{\ell} \Pr{\bm{S}_\ell = \bm{y}_n} %\sim \frac{(1 + \Exb{\zeta})^2 m }{n^2 2 \pi \sqrt{\det \bm{\Sigma}}} \exp\left(- \frac{x_\ell^2}{2 \sigma^2}\right)
		\sim \frac{(1 + \Exb{\zeta})^2 m }{n^2 2 \pi \sqrt{\det \bm{\Sigma}}} \exp\left(- \frac{x_\ell^2}{2 \sigma^2}\right).
	\end{align*}
	Let  $\varphi_{0, \sigma^2}(\cdot)$ denote the density of the centered normal distribution with variance $\sigma^2$. By Equation~\eqref{eq:nl} it follows that
	\begin{align}
		\label{eq:lltapp}
		\Pr{ \#\bm{T} = n, \#_1\bm{T} = \ell}  \sim  \frac{1}{n^2} \sqrt{\frac{1 + \Ex{\zeta}}{ 2 \pi \Va{\xi}}} m \varphi_{0, \sigma^2}(x_\ell)
	\end{align}
	uniformly for $|x_\ell| <M$. Setting 
	\begin{align}
		p_n := \frac{m}{\mathfrak{d}} n^{-3/2}\sqrt{\frac{1 + \Ex{\zeta}}{ 2 \pi \Va{\xi}}},
	\end{align}
	this entails that for any fixed real numbers $a_1<a_2$
	\begin{multline}
		\label{eq:cltdddd}
		p_n^{-1}\Prb{ \#\bm{T} = n, a_1 \le \frac{\#_1\bm{T} - n \mu}{\sqrt{n}} \le a_2} \\ \sim  \frac{\mathfrak{d}}{\sqrt{n}} \sum_{\substack{\ell \in j_n + \mathfrak{d} \ndZ \\  a_1 \le x_{\ell} \le a_2}}  \varphi_{0,\sigma^2}(x_{\ell}) \sim \int_{a_1}^{a_2} \varphi_{0, \sigma^2}(x) \,\text{d}x.
	\end{multline}
	We are going to argue that for each $\epsilon>0$ we may choose $M>0$ so that for all sufficiently large $n$
	\begin{align}
		\label{eq:toshow11}
		p_n^{-1} \Prb{ \#\bm{T} = n, |\#_1\bm{T} - n \mu| \ge M \sqrt{n} } \le \epsilon.
	\end{align}
	To this end, suppose that $1 \le \ell \le n$ satisfies $\ell \in j_n + \mathfrak{d} \ndZ$ and $|x_\ell| \ge M$. The multivariate local limit theorem in Proposition~\ref{pro:llt} entails that there is a constant $C>0$ that does not depend on $\ell$ such that 
	\begin{align}
		\Prb{ \bm{S}_\ell = \bm{y}_n} \le \frac{C}{n x_\ell^2 }.
	\end{align}
	Hence
	\begin{align}
		n^{3/2} \sum_{\substack{1 \le \ell \le n \\ |x_\ell| \ge M}} \frac{1}{\ell}\Prb{ \bm{S}_\ell = \bm{y}_n} \le \sqrt{n} \sum_{\substack{1 \le \ell \le n \\ |x_\ell| \ge M}} \frac{C}{\ell x_\ell^2}.
	\end{align}
	If we restrict to summands with $|x_\ell| \ge n^{1/4 + \delta}$ for any fixed $0<\delta<1/4$, we may be bound this by
	\begin{align}
		n^{-2\delta} \sum_{\substack{1 \le \ell \le n}} \frac{C}{\ell} = o(1).
	\end{align}
	If we restrict to summands with $M \le |x_\ell| \le n^{1/4 + \delta}$ instead, then $\ell = \Theta(n)$ and  we obtain the bound
	\begin{align}
		O(n^{-1/2}) \sum_{\substack{1 \le \ell \le n \\ M \le |x_\ell| \le n^{1/4 + \delta}}} \frac{1}{x_\ell^2} = O(1) \int_{M}^{\infty} \frac{1}{x^2} \,\text{d}x.
	\end{align}
	Taking $M$ large enough, this bound is smaller than $\epsilon/2$. This verifies Equation~\eqref{eq:toshow11}.
	
	Combining Equations~\eqref{eq:cltdddd} and ~\eqref{eq:toshow11} we obtain
	\begin{align}
		\label{eq:asymptoticaabT}
		\Pr{\# \bm{T} = n} \sim p_n = \frac{m}{\mathfrak{d}} \sqrt{\frac{1 + \Ex{\zeta}}{2 \pi \Va{\xi}}} n^{-3/2}.
	\end{align}
	Using~\eqref{eq:lltapp} it follows that
	\begin{align}
		\sqrt{n} \Pr{\#_1\bm{T} = \ell \mid \#\bm{T} = n} = \frac{\mathfrak{d}}{\sigma \sqrt{2\pi}} \exp \left( -\frac{x^2}{2 \sigma^2}\right) +  o(1)
	\end{align}
	uniformly for  $\ell \in j_n + \mathfrak{d} \ndZ$ with $|x_\ell| \le M$. Using~\eqref{eq:toshow11}, it follows that this remains true if we remove the condition that $|x_\ell|$ needs to be bounded.
	
	By~\eqref{eq:cltdddd} we obtain
	\begin{align}
		\frac{(\#_1\bm{T} \mid \#\bm{T} = n) - n \mu}{\sqrt{n}} \convdis \cN(0, \sigma^2).
	\end{align}
	It remains to show that the integer $\mathfrak{d}$ (as defined in this proof) may be chosen as in~\eqref{eq:doo}. That is, we have to show that
	\begin{align}
		\label{eq:do}
		\mathfrak{d} = m/D.
	\end{align} 
	Some jokes have been told about how mathematicians solve problems by reducing them to previously solved problems, even when it's not the most direct solution. The following short but not necessarily direct justification of~\eqref{eq:do} might fit into this category. For any sufficiently large $K>0$ there is a random variable $\bm{\xi}_1^{(K)} = (\xi^{(K)}, \zeta^{(K)})$ with support $\mathrm{supp}(\bm{\xi}_1^{(K)}) = \{ \bm{x} \in \mathrm{supp}(\bm{\xi}_1) \mid \| \bm{x} \|_2 < K\}$ that satisfies $\Ex{\xi^{(K)}}=1$. Let $\bm{T}^{(K)}$ denote the sesqui-type tree defined for $\bm{\xi}_1^{(K)}$ instead of $\bm{\xi}_1$. We assume that $K$ is sufficiently large so that the support of $\bm{\xi}_1^{(K)}$ is not contained on a straight line. Lemma~\ref{le:lemn} and everything we have shown so far applies to $\bm{T}^{(K)}$, yielding an analogon $D^{(K)}$ to $D$ that satisfies $D^{(K)} | D$, an analogon $\bm{D}^{(K)}$ to $\bm{D}$ that satisfies $ \bm{D}^{(K)} \ndZ^2 \subset \bm{D} \ndZ^2$, and an analogon $\mathfrak{d}^{(K)}$ to $\mathfrak{d}$ that satisfies $\mathfrak{d}^{(K)} | \mathfrak{d}$. The construction of these constants implies that for $K$ sufficiently large equality holds, meaning we have $\mathfrak{d}^{(K)} = \mathfrak{d}$, $D^{(K)} = D$, and may assume that $\bm{D} = \bm{D}^{(K)}$. As $\bm{\xi}^{(K)}_1$ is bounded and hence has finite exponential moments, we may apply a singularity analysis result by \cite[Thm. 28]{MR2240769}, yielding
	\begin{align}
		\Pr{\# \bm{T}^{(K)} = n} \sim  D \sqrt{\frac{1 + \Ex{\zeta^{(K)}}}{2 \pi \Va{\xi^{(K)}}}} n^{-3/2}
	\end{align}
	as $n \in a + D \ndZ$ tends to infinity. At the same time, Equation~\eqref{eq:asymptoticaabT} applied to $\bm{T}^{(K)}$ yields 
	\begin{align}
		\Pr{\# \bm{T}^{(K)} = n} \sim  \frac{m}{\mathfrak{d}} \sqrt{\frac{1 + \Ex{\zeta^{(K)}}}{2 \pi \Va{\xi^{(K)}}}} n^{-3/2}.
	\end{align}
	It follows that
	\[
	D = \frac{m}{\mathfrak{d}}.
	\]
	This completes the proof.
\end{proof}

In the case that $(\xi, \zeta)$ has finite exponential moments,  Equation~\eqref{eq:asymptoticT} may also be obtained using singularity analysis methods, see~\cite[Thm. 28]{MR2240769}. Here we only assumed a finite covariance matrix.

Having assured the concentration of the proportion of types in Equation~\eqref{eq:proportion}, we readily obtain by Theorem~\ref{te:convfringe3}:

\begin{theorem}
	\label{te:daone}
	Under the same assumptions as in Lemma~\ref{le:lembiv}, it follows that
	\begin{align}
		\mfL( (\bm{T}_n, v_n) \mid \bm{T}_n) \convp \mathfrak{L}(\hat{\bm{T}}(1, \eta))
	\end{align} 
	as $n \in a + D \ndN$ tends to infinity. Here $\eta$ denotes the independent Bernoulli-distributed random type from $\{1,2\}$ with distribution given by \begin{align}
		\Pr{\eta=1} = \frac{1}{1 + \Ex{\zeta}} \qquad \text{and} \qquad \Pr{\eta=2} = \frac{\Ex{\zeta}}{1 + \Ex{\zeta}}.
	\end{align}
\end{theorem}

See also~\cite[Thm. 27]{StEJC2018} for a similar result in the combinatorial setting of random unlabelled $\cR$-enriched trees where finite exponential moments were assumed.

\subsection{Critical reducible Galton--Watson trees}
Suppose that the type set is given by $\mathfrak{G}= \{1, \ldots, d\}$ for some integer $d \ge 2$.
Let $\bm{\sigma}$ be an ordered $d$-type offspring distribution and $\bm{\xi} = (\bm{\xi}_1, \ldots, \bm{\xi}_d)$ the corresponding unordered offspring distribution with $\bm{\xi}_i = (\xi_{i,1}, \dots, \xi_{i,d})$ for all $1 \le i \le d$. We assume that almost surely
\begin{align}
	\xi_{i,j} =0 \qquad \text{for all $i>j$}.
\end{align}
That is, a particle of type $i$ may only have offspring with types $j \ge i$. We furthermore assume that $\xi_{1,1}$ satisfies $\Pr{\xi_{1,1}=0}>0$, $\Pr{\xi_{1,1}\ge 2} >0$, and
\begin{align}
	\Ex{\xi_{1,1}} = 1.
\end{align}

Recall that for $\kappa := 1$ the random tree $\bm{T}^1$ is distributed like the $\bm{\sigma}$-Galton--Watson tree~$\bm{T}$, only that non-root vertices of type $1$ receive no offspring. We set
\begin{align}
	\xi := \#_1 \bm{T}^1-1 \qquad \text{and} \qquad \zeta := \sum_{i =2}^d \#_i \bm{T}^1,
\end{align}
so that $\xi + \zeta$ is the number of non-root vertices of $\bm{T}^1$ and $\xi \eqdist \xi_{1,1}$.

Let $\bm{S}$ denote a sesqui-type tree started at a type $1$ vertex, with vertices of the first type receiving offspring according to independent copies of $(\xi, \zeta)$, and vertices of the second type being infertile. It follows that
\begin{align}
	\label{eq:idid}
	\left(\#_1 \bm{T}, \sum_{i =2}^d \#_i \bm{T}\right) \eqdist (\#_1 \bm{S}, \#_2 \bm{S}).
\end{align}

We let $\bm{T}_n$ denote the result of conditioning $\bm{T}$ on having $n$ vertices. This is possible when $n \in a + D \ndZ$ is large enough, with $a$ and $D$ defined for $(\xi, \zeta)$ as in Lemma~\ref{le:lemn}. Equation~\eqref{eq:idid} allows us to transfer properties of sesqui-type trees:

%It is clear that properties for sesqui-type trees carry over:%, as $\#_1\bm{S} = \#_1\bm{T}$ and $\#\bm{S} = \#\bm{T}$.
\begin{lemma}
	\label{le:transferred}
	Suppose that additionally $(\xi,\zeta)$ has a finite covariance matrix and that its support is not contained in a straight line. Then Lemma~\ref{le:lembiv} holds analogously for the $d$-type reducible $\bm{\sigma}$-Galton--Watson tree $\bm{T}$, yielding an asymptotic expression for $\Pr{\# \bm{T} = n}$ and a local limit theorem for $\#_1 \bm{T}_n$.
\end{lemma}

In particular,
\begin{align}
	\frac{\#_1 \bm{T}_n}{n} \convp \frac{1}{1 + \Ex{\zeta}}.
\end{align}
We may also verify concentration for the remaining types:
\begin{lemma}
	Under the same assumptions as in Lemma~\ref{le:transferred}, it follows that
	\begin{align}
		\frac{\#_\gamma \bm{T}_n}{n} \convp \frac{\Ex{ \#_\gamma \bm{T}^1}}{1 + \Ex{\zeta}}
	\end{align}
	for all types $2 \le \gamma \le d$. Here we restrict $n$ to lie in the lattice $a + D \ndZ$.
\end{lemma}
\begin{proof}
	Let $\epsilon>0$ be given. Using the central limit theorem for $\#_1 \bm{T}_n$ from Lemma~\ref{le:transferred}, we may select $M>0$ large enough so that 
	\begin{align}
		\label{eq:notin}
		\Pr{\#_1 \bm{T}_n \notin I_n} < \epsilon
	\end{align}
	for $I_n := n/(1 + \Ex{\zeta}) \pm M \sqrt{n}$ and all large enough $n$. 	Let $\bm{T}^{1}_1, \bm{T}^{1}_2, \ldots$ denote independent copies of $\bm{T}^{1}$ and set for each $i \ge 1$
	\begin{align}
		\xi_i= \#_1 \bm{T}^1_i, \qquad \rho_i = \sum_{\substack{2  \le k \le d \\ k \ne \gamma}} \#_k \bm{T}^1_i, \qquad \text{and} \qquad \tau_i = \#_\gamma \bm{T}_i^1.
	\end{align}
	Let $\delta>0$ be given and set $J_n = (1 \pm \delta)\frac{\Ex{ \#_\gamma \bm{T}^1}}{1 + \Ex{\zeta}}n$. As $\Pr{\# \bm{T} =n} = \Theta(n^{-3/2})$, we may write
	\begin{align}
		\label{eq:cc1}
		\Pr{\#_1 \bm{T}_n \in I_n, \#_\gamma \bm{T}_n \notin J_n} &= O(n^{3/2}) \Pr{\#_1 \bm{T} \in I_n, \#_\gamma \bm{T} \notin J_n, \# \bm{T} = n}.
	\end{align}
	Arguing analogously as for $\eqref{eq:nlpre}$ and $\eqref{eq:nl}$, it follows that
	\begin{multline}
		\label{eq:cc2}
		\Pr{\#_1 \bm{T} \in I_n, \#_\gamma \bm{T} \notin J_n, \# \bm{T} = n} \\= \sum_{\ell \in I_n} \frac{1}{\ell} \Prb{\sum_{i=1}^\ell (\xi_i -1, \rho_i + \tau_i +1) = (-1,n),  \sum_{i=1}^\ell \tau_i \notin J_n }.
	\end{multline}
	Consider the collection 
	\begin{align}
		Q_n = \{(-1,a,b) \mid a,b \in \ndN_0, a+ b +1 = n,b \notin J_n\}.
	\end{align}
	This way,
	\begin{multline}
		\label{eq:cc3}
		\Prb{\sum_{i=1}^\ell (\xi_i -1, \rho_i + \tau_i +1) = (-1,n),  \sum_{i=1}^\ell \tau_i \notin J_n } \\= \Prb{\sum_{i=1}^\ell (\xi_i -1, \rho_i, \tau_i) \in Q_n}.
	\end{multline}
	Clearly $Q_n$ has at most $n$ elements, and each vector $(-1,a,b) \in Q_n$ satisfies
	\begin{align}
		\| (-1,a,b) - \ell (0, \Ex{\rho_1}, \Ex{\tau_1}) \|_2 \ge |b - \ell \Ex{\tau_1}| \ge \Theta(n)
	\end{align} 
	uniformly for $\ell \in I_n$. Using the local limit theorem in Proposition~\ref{pro:llt} for $3$ dimensions, it follows that
	\begin{align}
		\label{eq:cc4}
		\Prb{\sum_{i=1}^\ell (\xi_i -1, \rho_i, \tau_i) \in Q_n} 	&= o(n^{-3/2}).
	\end{align}
	Combining Equations~\eqref{eq:cc1},~\eqref{eq:cc2},~\eqref{eq:cc3} and~\eqref{eq:cc4} it follows that
	\begin{align}
		\Pr{\#_1 \bm{T}_n \in I_n, \#_\gamma \bm{T}_n \notin J_n} \to 0.
	\end{align}
	Using~\eqref{eq:notin} and the fact that $\epsilon>0$ was arbitrary, it follows that
	\begin{align}
		\Pr{\#_\gamma \bm{T}_n \notin J_n}  \to 0.
	\end{align}
\end{proof}

This allows us to apply Theorem~\ref{te:convfringe3}, yielding:

%For $\gamma \in \{1,2\}$ we may define the tree $\hat{\bm{S}}(1,\gamma)$ with a root vertex and an infinite backwards growing spine like the tree $\hat{\bm{T}}(1,\gamma)$, just for $\tilde{\bm{\sigma}}$ instead of $\bm{\sigma}$. See Remark~\ref{re:fu} for the case $\gamma=1$. Applying the transformation $\Xi$ to $\hat{\bm{S}}(1,1)$ yields a tree with a marked vertex and an infinite backwards growing spine consisting only of type $1$ vertices. If 

% We may construct the canonical decoration~$\bm{\beta}_{\hat{\bm{S}}(1, \gamma)}$ of $\hat{\bm{S}}(1,\gamma)$ analogously as for $\bm{S}$. There is a canonical bijection between the vertices of $\hat{\bm{S}}(1,\gamma)$ and its blow-up $\Xi\left(\hat{\bm{S}}(1, \gamma), \bm{\beta}_{\hat{\bm{S}}(1, \gamma)}\right)$, making $\Xi\left(\hat{\bm{S}}(1, \gamma), \bm{\beta}_{\hat{\bm{S}}(1, \gamma)}\right)$ a tree with a root and an infinite backwards growing spine.  
%It follows from Theorem~\ref{te:daone}:
\begin{theorem}
	\label{te:transferred}
	Let $\mathfrak{G}_0 \subset \{1, \ldots, d\}$ denote a non-empty subset.
	Let $v_n$ be uniformly selected among all vertices of $\bm{T}_n$ with type in $\mathfrak{G}_0$.
	We set $p(i) = \Ex{\#_i \bm{T}^1}$ for $i \in \mathfrak{G}_0\setminus\{1\}$, and $p(1) =1$ in case $1 \in \mathfrak{G}_0$.
	Under the same assumptions as in Lemma~\ref{le:transferred}, it follows that
	\begin{align}
		\mfL( (\bm{T}_n, v_n) \mid \bm{T}_n) \convp \mathfrak{L}(\hat{\bm{T}}(1, \eta))
	\end{align} 
	as $n \in a + D \ndN$ tends to infinity. Here $\eta$ denotes an independent random type from $\mathfrak{G}_0$, with distribution given by
	\begin{align}
		\Pr{\eta= j} = p(j) / \sum_{i  \in \mathfrak{G}_0} p(i), \qquad j \in \mathfrak{G}_0.
	\end{align}
\end{theorem}

\subsection{Regular critical irreducible Galton--Watson trees}
\label{sec:regcrit}

Suppose that the type set $\mathfrak{G}= \{1, \ldots, d\}$ is finite. Suppose that the ordered offspring distribution $\bm{\sigma}$ corresponds to the unordered offspring distribution  $\bm{\xi} = (\bm{\xi}_i)_{1 \le i \le d}$ with $\bm{\xi}_i = (\xi_{i,1}, \ldots, \xi_{i,d})$ for all $1 \le i \le d$.
The mean matrix
$
\bm{A} := (\Ex{\xi_{i,j}})_{1 \le i,j \le d}
$
is said to be \emph{finite}, if each coordinate is finite. We say $\bm{\xi}$ and the associated  branching process is \emph{irreducible}, if for any types $i,j$ there is an integer $k \ge 1$ such that the $(i,j)$th entry of $\bm{A}^k$ is positive. The Perron--Frobenius theorem ensures that in this case the spectral radius $\lambda$ of $\bm{A}$  is also an eigenvalue. If $\lambda=1$ (or $<1$ or $>1$), we say $\bm{\xi}$ / $\bm{\sigma}$ and the associated branching process is \emph{critical} (or subcritical or supercritical). By~\cite[Thm. 2 on page 186]{MR0373040}, a $\bm{\sigma}$-Galton--Watson tree $\bm{T}$ is almost surely finite (regardless with which type we start) if its critical or subcritical. If $\bm{\xi}$ is critical and has finite exponential moments, we say it is \emph{regular critical}.

Let $\bm{\omega} = (\omega_i)_{1 \le i \le d} \in \ndN_{0}^d \setminus \{\bm{0}\}$ be fixed. Given a type $1 \le \kappa \le d$ we are interested in the tree $\bm{T}_n$ obtained by conditioning the tree $\bm{T}$ started with a deterministic root type $\kappa$ on the event that
\begin{align}
	|\bm{T}|_{\bm{\omega}} := \sum_{i=1}^d \omega_i \,  \#_i \bm{T} = n.
\end{align}
It was shown by~\cite[Prop. 2.2]{MR3769811} that there is an integer $D \ge 1$ such that  $|\bm{T}|_{\bm{\omega}}$ is contained in some lattice of the form $\alpha_\kappa + D \ndZ$, and conversely any sufficiently large integer from the lattice is contained in the support. \cite[Sec. 4.3]{MR3769811} showed furthermore that if $\bm{\xi}$ is regular critical, then
\begin{align}
	\label{eq:obs1}
	\Pr{|\bm{T}|_{\bm{\omega}} = n} \sim c_{\bm{\omega}, \kappa} n^{-3/2}
\end{align}
for some explicitly given constant $c_{\bm{\omega}, \kappa}>0$ as $n \in \alpha_\kappa + D \ndZ$ tends to infinity. By \cite[Prop. 2.1]{MR3769811} (see also \cite{MR2469338}) it holds for any type $1 \le \gamma \le d$
\begin{align}
	\label{eq:obs2}
	\Ex{\#_\gamma \bm{T}^\gamma} =2
\end{align}
and $\#_\gamma \bm{T}^\gamma$ has finite exponential moments. Moreover, \cite[Lem. 6.7]{MR3769811} showed that for all integers $1 \le \gamma \le d$
\begin{align}
	\label{eq:obs3}
	\frac{\#_{\gamma} \bm{T}_n}{n} \convp c_{\gamma}
\end{align}
for a constant $c_{\gamma}>0$ that does not depend on $\kappa$. For any subset $\mathfrak{G}_0 \subset \{1, \ldots, d\}$ let $v_n$ be a uniformly selected vertex of $\bm{T}_n$ with type in $\mathfrak{G}_0$ and let $\eta$ denote an independent random type that assumes a type $\gamma$ with probability proportional to $c_{\gamma}$. Equations~\eqref{eq:obs1} and \eqref{eq:obs2} allow us to apply Theorem~\ref{te:convfringe}, yielding local convergence for the vicinity of a random vertex of type $\gamma$. Equation~\eqref{eq:obs3} entails consequently local convergence for the vicinity of $v_n$ to the corresponding mixture of limit objects:

\begin{theorem}
	\label{te:datwo}
	Suppose that $\bm{\xi}$ is regular critical. Then
	\begin{align}
		\mfL( (\bm{T}_n, v_n) \mid \bm{T}_n) \convp \mathfrak{L}(\hat{\bm{T}}(\eta))
	\end{align}
	as $n \in \alpha_\kappa + D \ndZ$ tends to infinity.  The limit object does not depend on the type $\kappa$ with which we started the branching process. Each type recurs almost surely infinitely many often along the backwards growing spine of the limit.
\end{theorem}

\begin{proposition}
	If  $|\mathfrak{G}_0|=1$ and if the vector $\bm{\omega}$ has only one non-zero coordinate, then Theorem~\ref{te:datwo} still holds if we only require $\bm{\xi}$ to be critical.
\end{proposition}
\begin{proof}
	Without loss of generality we may assume that $\bm{\omega} = (1,0,\ldots, 0)$. It holds that
	\begin{align}
		\Ex{\#_i \bm{T}^i} = 2
	\end{align}
	and
	\begin{align}
		\Ex{\#_i \bm{T}^j} < \infty
	\end{align}
	for all types $i,j$ by \cite[Prop. 2.1]{MR3769811}. If $\bm{T}$ starts with a vertex of type $1$, then $\#_1 \bm{T}$ is distributed like the population in a critical mono-type Galton--Watson tree with offspring distribution $\#_1 \bm{T}^1-1$. In particular, $\Pr{\#_1 \bm{T} = n} = \exp(o(n))$. If $\bm{T}$ starts with a vertex of a different type, then  $\#_1 \bm{T}$  is distributed like the sum of a random number of independent copies of populations of $(\#_1 \bm{T}^1-1)$-branching processes. Also in this case, $\Pr{\#_1 \bm{T}=n} = \exp(o(n))$.
	
	For each $i$ let $v_n^i$ denote a uniformly selected vertex of type $i$ of $\bm{T}_n$.  Since $\#_1 \bm{T}_n = n$ and $\bm{T}_n$ is obtained from $\bm{T}$ by conditioning on an event with probability $\exp(o(n))$, we may apply
	Theorem~\ref{te:convfringe} to obtain
	\begin{align}
		\mfL( (\bm{T}_n, v_n^1) \mid \bm{T}_n) \convp \mathfrak{L}(\hat{\bm{T}}(1)).
	\end{align} 
	Let $\alpha$ denote the deterministic root type with which we started the branching process. Let $\gamma \ne 1$ be a type. It follows from Equations~\eqref{eq:red2mono} and \eqref{eq:oins} that
	\begin{align}
		\Pr{\#_\gamma \bm{T}_n = \ell } = \frac{\Pr{\#_\gamma \bm{T}^\alpha + \sum_{i=1}^{n - \one_{\alpha=1}} \#_\gamma \bm{T}_i^1 = \ell, \#_1 \bm{T} = n} }{\Pr{\#_1 \bm{T} = n}},
	\end{align}
	with $(\bm{T}^1_i)_{i \ge 1}$ denoting independent copies of $\bm{T}^1$. Truncating the summands and using the Azuma--Hoeffding inequality (and again the fact that $\Pr{\#_1 \bm{T}=n} =  \exp(o(n))$), it follows that for any integer $D \ge 1$ and any $\epsilon>0$ it holds with high probability that
	\[
	\frac{\#_\gamma \bm{T}_n}{n} \ge \Ex{\#_1 \bm{T}_i^1, \#_1 \bm{T}_i^1 \le D} - \epsilon.
	\]
	As this holds for arbitrarily large $D$ and as $\Ex{\#_1 \bm{T}_i^1} < \infty$, it follows that
	\begin{align}
		\frac{\#_\gamma \bm{T}_n}{n} \ge \Ex{\#_1 \bm{T}_i^1} - o_p(1).
	\end{align}
	This allows us to apply~Theorem~\ref{te:convfringe} for $\kappa= \gamma$ to obtain
	\begin{align}
		\mfL( (\bm{T}_n, v_n^\gamma) \mid \bm{T}_n) \convp \mathfrak{L}(\hat{\bm{T}}(\gamma)).
	\end{align} 
	%	Alternatively, we could also have applied Proposition~\ref{pro:strengthened} to obtain
	%	\begin{align}
	%		\mfL( (\bm{T}_n, v_n^\gamma) \mid \bm{T}_n) \convp \mathfrak{L}(\hat{\bm{T}}(1,\gamma)) ).
	%	\end{align} 
	%	In the present setting (but not in general) it holds that $\hat{\bm{T}}(1,\gamma) \eqdist \hat{\bm{T}}(\gamma)$, so there is no difference.
\end{proof}

\subsection{Critical Galton--Watson trees conditioned by typed populations}

Let the type set be given by $\mathfrak{G} = \{1, \ldots, d\}$ and suppose that the unordered offspring distribution $\bm{\xi}$ corresponding to $\bm{\sigma}$ is irreducible and critical. Let $\Lambda \subset \ndN$ be an infinite subset. Let
\begin{align}
	\bm{k}(n) = (k_1(n), \ldots, k_d(n))
\end{align}
be a sequence in $\bm{\ndN}_0^d$ satisfying
\begin{align}
	\Pr{(\#_i \bm{T})_{1 \le i \le d} = \bm{k}(n)} > 0
\end{align}
for all $n \in \Lambda$. We assume that
\begin{align}
	\label{eq:aso}
	\frac{\bm{k}(n)}{\| \bm{k}(n)\|_1} \to \bm{a}
\end{align}
for some $\bm{a} = (a_i)_{1 \le i \le d} \in \ndR^d_{\ge 0} \setminus \{\bm{0}\}$ as $n \in \Lambda$ tends to infinity.
Let $\alpha$ be a random type. We let $\bm{T}_n$ denote the result of conditioning the $\bm{\sigma}$-Galton--Watson tree started with a particle of type $\alpha$ on the event
\begin{align}
	\label{eq:cond}
	(\#_i \bm{T})_{1 \le i \le d} = \bm{k}(n).
\end{align}
Let $1 \le \kappa \le d$ be a coordinate satisfying $a_\kappa >0$.  By \cite[Prop. 2.1]{MR3769811} (see also \cite{MR2469338}) it holds 
\begin{align}
	\label{eq:obs5}
	\Ex{\#_\kappa \bm{T}^\kappa} = 2.
\end{align}
It follows by Equations~\eqref{eq:red2mono} and~\eqref{eq:oins} that the event~\eqref{eq:cond} implies that a random number of critical $(\#_\kappa \bm{T}^\kappa-1)$-Galton--Watson trees has total size $k_\kappa(n)$. The total population of a critical mono-type Galton--Watson tree assumes an integer $n$ with probability $\exp(o(n))$, yielding for any $\epsilon>0$
\begin{align}
	\exp(- \epsilon k_\kappa(n)) = o(\Pr{(\#_i \bm{T})_{1 \le i \le d} = \bm{k}(n)}).
\end{align}
This verifies Assumption~\eqref{eq:boundassumption}. Let $v_n^\kappa$ denote a uniformly selected vertex of type $\kappa$ of the tree $\bm{T}_n$. It follows by Theorem~\ref{te:convfringe} that 
\begin{align}
	\mfL( (\bm{T}_n, v_n^\kappa) \mid \bm{T}_n) \convp \mathfrak{L}(\hat{\bm{T}}(\kappa)).
\end{align} 
Using~\eqref{eq:aso}, we deduce:
\begin{theorem}
	\label{te:dathree}
	Let $\mathfrak{G}_0 \subset \mathfrak{G}$ be a subset such that $a_i>0$ for at least one type $i \in \mathfrak{G}_0$.  Let $\eta$ be an independent random type from $\mathfrak{G}_0$ drawn with probability $\Pr{\eta= i} = a_i / \sum_{j \in \mathfrak{G}_0} a_j$.  Let $v_n$ denote a vertex of $\bm{T}_n$ that is uniformly selected among all vertices with type in $\mathfrak{G}_0$. Then
	\begin{align}
		\mfL( (\bm{T}_n, v_n) \mid \bm{T}_n) \convp \mathfrak{L}(\hat{\bm{T}}(\eta)).
	\end{align} 
\end{theorem}
See~\cite[Thm. 3.1]{MR3803914} for a limit in a similar setting that describes the asymptotic vicinity of the root.

\section*{Acknowledgement}
I  thank the referees  for the thorough reading, helpful recommendations, and corrections.


%\begin{thebibliography}{}
% etc
%\end{thebibliography}

\end{document}